\documentclass[10pt, a4paper]{amsart} %\documentclass[Ñ¡Ïî±í]{Àà}[°æ±ŸºÅyyyy/mm/dd]
\usepackage[left=2.50cm, right=2.50cm, top=2.50cm, bottom=2.50cm]{geometry} 
\usepackage{amscd, amsfonts, amsmath, amssymb, amsthm, arydshln, fixmath, graphicx, mathrsfs, overpic, tikz, tikz-cd}
\usepackage{hyperref}
\usepackage[all]{xy} %\usepackage[Ñ¡Ïî±í]{ºê°ü}[°æ±ŸºÅyyyy/mm/dd] ÊýÑ§×ÖÌåºê°ü:mathrsfs, amsfonts, amssymb
\usepackage{tikz}
\usepackage{pgfplots}
\usepackage{pgflibraryarrows}
\usepackage{pgflibrarysnakes}
\usepackage{dsfont}
\makeindex

\theoremstyle{definition} %±êÌâÓë±àºÅÎªºÚÌå, ÕýÎÄÎªÕý³£×ÖÌå
\newtheorem{Unity}{Unity}[section] %\newtheorem{¶šÀí»·Ÿ³Ãû}{±êÌâ}[Ö÷ŒÆÊýÆ÷Ãû]
\newtheorem*{Definition*}{Definition} %\newtheorem*{¶šÀí»·Ÿ³Ãû}[ÒÑ¶šÒå¶šÀí»·Ÿ³Ãû]{±êÌâ} ÊÖ¶¯±àºÅ, ²»×Ô¶¯±àºÅ
\newtheorem{Definition}[Unity]{Definition} %\newtheorem{¶šÀí»·Ÿ³Ãû}[ÒÑ¶šÒå¶šÀí»·Ÿ³Ãû]{±êÌâ} Óëµ±Ç°»·Ÿ³¹²ÓÃÍ¬Ò»žöÐòºÅŒÆÊýÆ÷

\theoremstyle{plain} %±êÌâÓë±àºÅÎªºÚÌå, ÕýÎÄÎªÐ±Ìå
\newtheorem*{Theorem*}{Theorem}
\newtheorem{Theorem}[Unity]{Theorem}
\newtheorem{Proposition}[Unity]{Proposition}
\newtheorem{Corollary}[Unity]{Corollary}
\newtheorem{Lemma}[Unity]{Lemma}

\theoremstyle{remark} %±êÌâÓë±àºÅÎªÐ±Ìå, ÕýÎÄÎªÕý³£×ÖÌå
\newtheorem*{Remark*}{Remark}
\newtheorem{Remark}[Unity]{Remark}
\newtheorem{Example}[Unity]{Example}

\numberwithin{Unity}{section}%\numberwithin{ŒÆÊýÆ÷}{Ö÷ŒÆÊýÆ÷}

\newcommand{\rk}{\mathrm{rk\,}}

\newcommand{\Ima}{\mathrm{Im\,}}

\newcommand{\End}{\mathrm{End}}

\newcommand{\Hom}{\mathrm{Hom}}

\newcommand{\Rep}{\mathrm{Rep}}

\newcommand{\Spec}{\mathrm{Spec\,}}

\newcommand{\Qcoh}{\mathfrak{Qcoh}}
\newcommand{\Vect}{\mathfrak{Vect}}

\newcommand{\coker}{\mathrm{coker\,}}
\newcommand{\Gal}{\mathrm{Gal}}
\newcommand{\Ext}{\mathrm{Ext}}

\newcommand{\characteristic}{\mathrm{char\,}}

\begin{document}

\title{Base Change Of Fundamental Group Schemes}
\author{Lingguang Li}
\address{School of Mathematical Sciences,
Key Laboratory of Intelligent Computing and Applications (Tongji University), Ministry of Education, Shanghai 200092, CHINA}
\email{LiLg@tongji.edu.cn}
\author{Niantao Tian}
\address{School of Mathematical Sciences,
Key Laboratory of Intelligent Computing and Applications (Tongji University), Ministry of Education, Shanghai 200092, CHINA}
\email{tianniantao@tongji.edu.cn}
\begin{abstract}
Let $K/k$ be a field extension, $X$ a connected quasi-compact and quasi-separated scheme over $k$ with a rational point $x$, $p:X_K\to X$ the natural projection, $\mathcal{C}_X$ and $\mathcal{C}_{X_K}$ the Tannakian categories whose objects consist of vector bundles on $X$ and $X_K$ respectively. We give criteria for the natural homomorphism $\pi(\mathcal C_{X_K},x_K)\to\pi(\mathcal C_X,x)_K$ to be an isomorphism. Suppose further $X$ is geometrically reduced connected scheme proper over $k$. As applications, we prove base change isomorphisms for $\pi^{\acute{e}t},\pi^{E\acute{e}t}$ under arbitrary extensions and for $\pi^S,\pi^N,\pi^{EN}$ under arbitrary algebraic extensions; a cuspidal cubic curve example shows that it can fail for purely transcendental extensions. In positive characteristic, base change isomorphism holds for $\pi^{Loc}$ and $\pi^{ELoc}$ under separable algebraic extensions; an elliptic curve example shows that it can fail for finite purely inseparable extensions.
\end{abstract}
\maketitle
\tableofcontents

\section{Introduction}

By the Tannakian duality of Saavedra Rivano \cite{Saa72} and Deligne--Milne \cite{DeMi82}, a neutral Tannakian category with a fibre functor determines an affine group scheme. Many fundamental group schemes arise in this way from categories of vector bundles on a pointed scheme $(X,x)$. For a field extension $K/k$, the base change problem asks when the natural homomorphism $\pi(\mathcal C_{X_K},x_K)\rightarrow\pi(\mathcal C_X,x)_K$ is faithfully flat or an isomorphism.

Grothendieck \cite{Gro60} defined the \'etale fundamental group using finite \'etale coverings and proved that, for $X$ proper over an algebraically closed field $k$ and an algebraically closed extension $K/k$, the natural homomorphism $\pi_1^{\acute{e}t}(X_K,\bar{x}_K) \rightarrow\pi_1^{\acute{e}t}(X,\bar{x})$ is an isomorphism.

Nori \cite{Nor82} introduced the fundamental group scheme $\pi^N(X,x)$ as the projective limit of pointed finite torsors, allowing arbitrary finite group schemes. For a geometrically reduced connected scheme proper over $k$, it is equivalently defined by the Tannakian category of essentially finite vector bundles. He proved that $\pi^N(X_K,x_K)\rightarrow\pi^N(X,x)_K$ is an isomorphism for separable algebraic extensions $K/k$ and asked whether this holds for arbitrary field extensions. Mehta and Subramanian \cite{MeSu02} gave a counterexample on an integral projective curve with a cuspidal singularity under an extension of algebraically closed fields. Pauly \cite{Pau07} later gave a smooth counterexample. St\"abler \cite{Sta14} constructed smooth projective counterexamples in every positive characteristic, using families of syzygy bundles to obtain stable Frobenius trivial bundles which do not descend to the original algebraically closed field.

Nori \cite{Nor82} also introduced the unipotent fundamental group scheme \(\pi^{uni}(X,x)\), defined by the Tannakian category of unipotent vector bundles on $X$. In contrast, he proved that if \(X\) is a scheme of finite type over \(k\) and \(H^0(X,\mathcal{O}_X)=k\), then $\pi^{uni}(X_K,x_K)\rightarrow \pi^{uni}(X,x)_K$ is an isomorphism for any field extension \(K/k\).

In positive characteristic, further fundamental group schemes were introduced to capture Frobenius phenomena. Mehta and Subramanian \cite{MeSu08} defined the local fundamental group scheme \(\pi^{Loc}(X,x)\) using Frobenius trivial vector bundles, and proved that its base change may fail under extensions of algebraically closed fields, while also giving equivalent conditions for the base change homomorphism to be an isomorphism.

Langer \cite{Lan11} introduced the S-fundamental group scheme $\pi^S(X,x)$, whose Tannakian category consists of numerically flat vector bundles on $X$. For a complete variety $X$ over an algebraically closed field $k$ and an algebraically closed extension $K/k$, he proved that the natural homomorphism $\pi^S(X_K,x_K)\rightarrow\pi^S(X,x)_K$ is faithfully flat \cite[Lemma~8.4]{Lan11}, although it need not be an isomorphism. Proposition~\ref{algfaith} recovers this faithful flatness result from our general Tannakian criterion and treats the other fundamental group schemes listed there in the same framework, with $X$ allowed to be a reduced connected proper scheme.

Otabe \cite{Ota17} introduced the EN-fundamental group scheme in characteristic zero, defined by semi-finite vector bundles, and proved a base change isomorphism after extending the ground field to its algebraic closure.

The purpose of this paper is to study these base change problems through general criteria for compatible Tannakian categories of vector bundles over arbitrary ground fields. For finite field extensions, our criteria characterize base change isomorphisms in terms of preservation of the relevant categories under finite pushforward. For saturated categories over algebraically closed base field, the isomorphism criterion reduces to the descent of simple objects. We also construct a filtration associated with finite purely inseparable extensions and use it to establish algebraic base change results for the fundamental group schemes considered below. These methods recover known cases and yield results in the generality specified in our statements.

The main results of this paper are the following Theorems and Propositions:

\begin{Theorem}[{Theorem~\ref{fieldgeneralcriteria}, Theorem~\ref{fieldfinite}, Theorem~\ref{fieldalgclosedthm}}]
    Let $k$ be a field, $K/k$ a field extension, $X$ a connected quasi-compact and quasi-separated scheme over $k$, $p:X_K\to X$ the natural projection, $x_K\in X_K(K)$ lying over $x\in X(k)$, $\mathcal{C}_X$ and $\mathcal{C}_{X_K}$ the Tannakian categories whose objects consist of vector bundles on $X$ and $X_K$ respectively. Suppose for any $E\in\mathcal C_X$, we have $E\otimes_k K\in\mathcal C_{X_K}$. Then
    \begin{enumerate}
        \item The following conditions are equivalent:
        \begin{enumerate}
            \item The natural homomorphism $\pi(\mathcal{C}_{X_K},x_K)\rightarrow \pi(\mathcal{C}_X,x)_K$ is an isomorphism.
            \item For any $\mathcal{E}\in\mathcal{C}_{X_K}$, there exists $ E\in\mathcal{C}_X$ such that $\mathcal{E}\hookrightarrow E\otimes_k K$.
            \item For any $\mathcal E\in\mathcal C_{X_K}$, there is a filtered colimit system $\{\mathcal F_\lambda\}_{\lambda\in \Lambda}$ of subsheaves of $ p_*\mathcal E$ such that $\mathcal F_\lambda\in\mathcal C_X$ for every $\lambda$ and $p_*\mathcal E=\varinjlim_\lambda\mathcal F_\lambda$ as $\mathcal O_X$-modules.
        \end{enumerate}
        \item Let $K/k$ be finite. Then the following conditions are equivalent:
        \begin{enumerate}
  \item For any $\mathcal E\in\mathcal C_{X_K}$, we have $p_*\mathcal E\in\mathcal C_X$.
  \item The natural homomorphism $\pi(\mathcal C_{X_K},x_K)\rightarrow\pi(\mathcal C_X,x)_K$ is an isomorphism.
  \item \begin{enumerate}
            \item For any $E\in\Vect(X)$, $E\in\mathcal C_X$ iff $E\otimes_k K\in\mathcal C_{X_K}$.
            \item For any $\mathcal E\in\mathcal C_{X_K}$, we have $p^*p_*\mathcal E\in\mathcal C_{X_K}$.
  \end{enumerate}
\end{enumerate}
        \item Let $k$ be algebraically closed, $\mathcal{C}_X$ and $\mathcal C_{X_K}$ saturated. Then the following conditions are equivalent: 
        \begin{enumerate}
            \item The natural homomorphism $\pi(\mathcal{C}_{X_K},x_K)\rightarrow \pi(\mathcal{C}_X,x)_K$ is an isomorphism.
            \item For any simple object $\mathcal{E}\in\mathcal{C}_{X_K}$, there exists $E\in\mathcal{C}_X$ such that $\mathcal{E}\cong E\otimes_k K$.
        \end{enumerate}
    \end{enumerate}
\end{Theorem}

We also give more easily verifiable necessary and sufficient conditions for base change isomorphisms under finite Galois and finite purely inseparable extensions; see Proposition~\ref{fieldfiniteequiv}.

Applying the above theorems, we generalize base change results for the S, Nori, EN, \'et, E\'et, Loc and ELoc fundamental group schemes under different types of field extensions:

\begin{Theorem}[Theorem~\ref{algebraicbasechange}, Proposition~\ref{etalealgclosed}]
    Let $k$ be a field, $K/k$ a field extension, $X$ a geometrically reduced connected scheme proper over $k$, $x_K\in X_K(K)$ lying over $x\in X(k)$. Then 
    \begin{enumerate}
        \item For $*\in\{\acute{e}t,E\acute{e}t\}$, the natural homomorphism $\pi^{*}(X_K, x_K)\rightarrow \pi^{*}(X,x)_K$ is an isomorphism.
        \item Let $K/k$ be algebraic and $*\in\{S,N,\acute{e}t,EN,E\acute{e}t\}$. The natural homomorphism $\pi^*(X_K,x_K)\rightarrow\pi^*(X,x)_K$ is an isomorphism.
        \item Let $K/k$ be algebraic separable and $*\in\{Loc,ELoc\}$. The natural homomorphism $\pi^*(X_K,x_K)\rightarrow\pi^*(X,x)_K$ is an isomorphism.
    \end{enumerate}
\end{Theorem}

In Example~\ref{puretranscendentalcounterexample}, we construct a counterexample for base change of S and Nori fundamental group schemes under purely transcendental extension. In Example~\ref{ex:bc-local-counterexample}, we construct a counterexample for base change of Loc and ELoc fundamental group schemes under purely inseparable extension. Failures of base change under extensions of algebraically closed fields are discussed in Corollary~\ref{5.47} and Proposition~\ref{containSnobasechange}.

Section~2 collects the necessary facts on Tannakian categories and proves that their saturation is again Tannakian. Section~3 establishes the general base change criteria. Section~4  applies the criteria to the fundamental group schemes above.

\section{Preliminaries}

\textbf{Notation and Convention:} Let $k$ be a field, $X$ a scheme over $k$, $x\in X(k)$, $G$ an affine group scheme over $k$, $\mathfrak{Qcoh}(X)$ the category of quasi-coherent sheaves on $X$, $\mathfrak{Vect}(X)$ the category of vector bundles on $X$. Consider the following Cartesian diagrams
\[
    \begin{aligned}
        \begin{tikzcd}
            X\times_{\Spec k}\Spec K\arrow[r,"p"]\arrow[d]& X\arrow[d]\\
            \Spec K\arrow[r]&\Spec k
        \end{tikzcd}&&\begin{tikzcd}
            G\times_{\Spec k}\Spec K\arrow[r,"p"]\arrow[d]& G\arrow[d]\\
            \Spec K\arrow[r]&\Spec k
        \end{tikzcd}
    \end{aligned}
\]
We denote $X_K := X \times_{\Spec k} \Spec K$ and $G_K := G \times_{\Spec k}\Spec K$ for the base changes, and $p: X_K \to X$ is the canonical projection and $x_K \in X_K(K)$ is a $K$-valued point lying over $x\in X(k)$. For a coherent sheaf $E$ on $X$, we denote its pullback $p^*E$ simply by $E \otimes_k K$. For a category $\mathcal{V}_X$ consisting of vector bundles on $X$, a nonzero object $E\in\mathcal{V}_X$ is said to be a simple object if it has no nonzero proper subobject in $\mathcal{V}_X$ and to be a semisimple object if it is isomorphic to a direct sum of simple objects in $\mathcal V_X$. 

Throughout this paper, a Tannakian category on $X$ means a neutral Tannakian category over $k$, realized as a strictly full tensor subcategory $\mathcal C_X$ of $\Vect(X)$, with fibre functor $|_x:\mathcal C_X\to\mathrm{Vect}_k$, $E\mapsto E|_x$, such that the inclusion $\mathcal C_X\hookrightarrow\Qcoh(X)$ is exact. Its Tannaka group scheme is denoted by $\pi(\mathcal C_X,x)$. A Tannakian subcategory of another Tannakian category is understood to be closed under subquotients in that category.

\begin{Definition}
    Let $k$ be a field, $\mathcal{C},\mathcal{D}$ Tannakian categories over $k$, $\Phi:\mathcal{C}\rightarrow \mathcal{D}$ an exact tensor functor. The functor $\Phi$ is said to be \textit{observable} if for any $E\in\mathcal{C}$ and any subobject $L\hookrightarrow \Phi(E)\in\mathcal{D}$ of rank 1, there exists $F\in\mathcal{C}$ and integer $n>0$ such that $(L^{\otimes n})^{\vee}\hookrightarrow \Phi(F)$ is a subobject in $\mathcal{D}$.
\end{Definition}

\begin{Lemma}\label{Observable}
    Let $k$ be a field, $\mathcal{C},\mathcal{D}$ neutral Tannakian categories over $k$ with fibre functor $\omega_x$, $\Phi:\mathcal{C}\rightarrow \mathcal{D}$ an exact tensor functor. 
Then the following conditions are equivalent:
    \begin{enumerate}
        \item The functor $\Phi$ is observable.
        \item For any $E\in\mathcal{C}$ and any subobject $L\hookrightarrow \Phi(E)\in\mathcal{D}$ of rank 1, there exists $F\in\mathcal{C}$, s.t. $L^{\vee}\hookrightarrow \Phi(F)\in\mathcal{D}$.
        \item For any $E\in\mathcal{C}$ and any subobject $E'\hookrightarrow\Phi(E)\in\mathcal{D}$, there exists $F\in \mathcal{C}$, s.t. ${E'}^\vee\hookrightarrow \Phi(F)\in\mathcal{D}$.
        \item For any $E\in\mathcal{C}$ and any quotient $\Phi(E)\twoheadrightarrow E'\in\mathcal{D}$, there exists $F\in \mathcal{C}$, s.t. $E'\hookrightarrow \Phi(F)\in\mathcal{D}$.
        \item For any $E\in\mathcal{C}$ and any subobject $E'\hookrightarrow\Phi(E)\in\mathcal{D}$, there exists $F\in \mathcal{C}$, s.t. $\Phi(F)\twoheadrightarrow E' \in\mathcal{D}$.
    \end{enumerate}
\end{Lemma}

\begin{proof}
    $(1)\Leftrightarrow (2)$ By the isomorphism $L\otimes L^\vee\cong \mathds{1}_{\mathcal{D}}$, it follows immediately.

    $(1)\Leftrightarrow (5)$ It follows by \cite[Proposition~A.3]{DaEs22}.
    
    $(3)\Leftrightarrow (4)$ Taking duality, it follows immediately.

    $(4)\Leftrightarrow (5)$ Taking duality, it follows immediately.
\end{proof}

Let $k$ be a field, $X$ a connected scheme over $k$, $x\in X(k)$, $\mathcal{C}_X$ a Tannakian category whose objects consist of vector bundles on $X$ with neutral fibre functor $|_x$, and $\pi(\mathcal{C}_X,x)$ its corresponding Tannaka group scheme. For any subset $M$ of $\mathcal{C}_X$, we denote $$M_1:=\{E~|~E\in M\text{ or }E^\vee\in M\}\text{, }M_2:=\{~E_1\otimes E_2\otimes\cdots \otimes E_m~|~E_i\in M_1, 1\leq i\leq m, m\in\mathbb{N}\},$$
$$\langle M\rangle_{\mathcal{C}_X}:=
\left\{ 
    E\in \mathcal{C}_X \left\vert 
    \begin{split}
        & \exists F_i\in M_2,1\leq i\leq t, \text{ and } E_1,E_2\in \mathcal{C}_X\\
        & s.t. \text{ }E_1\hookrightarrow E_2\hookrightarrow\bigoplus^{t}_{i=1}F_i,\text{ and } E\cong E_2/E_1
    \end{split}\right.
\right\}.
$$
Then $(\langle M\rangle_{\mathcal{C}_X},\otimes,|_x,\mathcal{O}_X)$ is a Tannakian subcategory of $(\mathcal{C}_X,\otimes,|_x,\mathcal{O}_X)$ and we denote its Tannaka group scheme by $\pi(\langle M\rangle_{\mathcal{C}_X},x)$. Then we have:
\begin{itemize}
      \item[(a)] For any subsets $N\subseteq M\subseteq \mathcal{C}_X$,
          there is a natural surjective homomorphism of affine group schemes $$\rho_{N}^M: \pi(\langle M\rangle_{\mathcal{C}_X},x)\twoheadrightarrow \pi(\langle N\rangle_{\mathcal{C}_X},x).$$
      \item[(b)] $\pi(\mathcal{C}_X,x)$ is the inverse limit of $\pi(\langle M\rangle_{\mathcal{C}_X},x)$, where $M$ runs through all subsets of $\mathcal{C}_X$.
\end{itemize}

Let $k$ be a field, $G$ an affine group scheme over $k$, $X$ a scheme over $k$. Denote the finite dimensional representation category of $G$ by $\Rep_k^f(G)$. The regular representation of $G$ on $k[G]$ is, for any $g,h\in G$ and $f\in k[G]$, we have $(g\cdot f)(h)=f(hg)$. Given a principal $G$-bundle $\phi:P\rightarrow X$, define an exact tensor functor 
$$\begin{aligned}
    \eta_k^P:\Rep_k^f(G)\rightarrow \Qcoh(X),&V\mapsto (\mathcal{O}_P\otimes_k V)^G, 
\end{aligned}$$
where the action is the diagonal action. Conversely, given an exact tensor functor $\eta_k: \Rep_k^f(G)\rightarrow \Qcoh(X)$, one can define a principal $G$-bundle $\pi:P\rightarrow X$ by $P:=\Spec (\eta_k(k[G]))$,
$$\eta_k(k[G]):=\varinjlim\limits_{\alpha\in I}\eta_k(V_\alpha),$$
where $V_\alpha$ runs through all $G$-invariant finite dimensional subspaces of $k[G]$.

Consequently, there is a one-to-one correspondence between principal $G$-bundles $\phi:P\rightarrow X$ and exact tensor functors $\eta_k^P:\Rep_k^f(G)\rightarrow \Vect(X)$. In particular, consider a Tannakian subcategory $\mathcal{C}_X$ of $\Vect(X)$, $\pi(\mathcal{C}_X,x)$ its Tannaka group scheme, then the embedding functor $\eta_k:\Rep_k^f(\pi(\mathcal{C}_X,x))\cong \mathcal{C}_X\hookrightarrow \Vect(X)$ is a fully faithful exact tensor functor and corresponds to a universal principal $\pi(\mathcal{C}_X,x)$-bundle $\phi:P\rightarrow X$.

\begin{Lemma}[{\cite[Proposition~A.6]{DaEs22}}]\label{Observablesurjective}
    Let $k$ be a field, $f:G\rightarrow H$ a homomorphism of affine group schemes over $k$, $f^*:\Rep_k^f(H)\rightarrow \Rep_k^f(G)$ the functor induced by $f$. Then the following conditions are equivalent:
    \begin{enumerate}
        \item $f$ is faithfully flat.
        \item $f^*$ is fully faithful and observable.
        \item $f^*(\Rep_k^f(H))\subseteq \Rep_k^f(G)$ is a full subcategory stable under the operation of taking subobjects.
        \item For every $V \in \Rep_k^f(H)$, the inclusion $\mathbb{P}(V)^H(k) \subseteq \mathbb{P}(V)^G(k)$ is an equality.
    \end{enumerate}
\end{Lemma}

\begin{Lemma}[{\cite[Chapter~VII, Proposition~2.3 \& Proposition~7.2]{Mil12}}]\label{isofield}
    Let $k$ be a field, $K/k$ a field extension, $f:G\rightarrow H$ a homomorphism of affine group schemes over $k$. Then $f$ is an isomorphism iff $f_K:G_K\rightarrow H_K$ is an isomorphism.
\end{Lemma}

\begin{Lemma}\label{NoetherDeuring}
    Let $k$ be a field, $K/k$ a field extension, $G$ an affine group scheme over $k$. For any $V,W\in\Rep_k^f(G)$, if $V\otimes_k K\cong W\otimes_k K$ as finite dimensional $G_K$ representations, then $V\cong W$ as finite dimensional $G$ representations.
\end{Lemma}

\begin{proof}
    It follows by Noether-Deuring theorem (cf. \cite[Theorem~2.2]{BeRe19}).
\end{proof}

\begin{Lemma}[{\cite[Chapter VIII, Proposition 10.2]{Mil12}}]\label{regularrepresentation}
    Let $k$ be a field, $G$ an affine group scheme over $k$, $(V, \rho)\in \Rep_k^f(G)$. Then $V$ embeds into a finite sum of copies of the regular representation, i.e. there exists integer $n>0$ and a $G$-invariant homomorphism $V\hookrightarrow k[G]^{\oplus n}$ where $k[G]$ is the regular representation.
\end{Lemma}

\begin{Lemma}\label{basechangeembedding}
    Let $k$ be a field, $K/k$ a field extension, $G$ an affine group scheme over $k$. For any $V\in\Rep_K^f(G_K)$, there exists $W\in\Rep_k^f(G)$ such that $V\hookrightarrow W\otimes_k K$ in $\Rep_K^f(G_K)$.
\end{Lemma}

\begin{proof}
    For any $V\in \Rep_K^f(G_K)$, by Lemma~\ref{regularrepresentation} we have $$V\hookrightarrow K[G_K]^{\oplus n}\cong k[G]^{\oplus n}\otimes_k K.$$
    Then there exists a finite dimensional $k$-subspace $W'\hookrightarrow k[G]^{\oplus n}$ such that $V\hookrightarrow W'\otimes_k K$ as $K$-vector space. There exists $W\in\Rep_k^f(G)$ and a $G$-invariant homomorphism $W\hookrightarrow k[G]^{\oplus n}$ such that $W$ contains $W'$, so $V\hookrightarrow W\otimes_k K$ as finite dimensional $K$-representations of $G_K$.
\end{proof}

\begin{Theorem}[{\cite[Theorem~A.1]{EHS07}}]\label{Thm1}
Let $L\xrightarrow{q}G\xrightarrow{p} A$ be a sequence of homomorphisms of affine group schemes over a field $k$. It induces a sequence of functors:
$$\Rep_k^f(A)\xrightarrow{p^*}\Rep_k^f(G)\xrightarrow{q^*}\Rep_k^f(L),$$
where $\Rep_k^f$ denotes the category of finite dimensional representations over $k$. Then we have the following:
\begin{enumerate}
    \item[(1)] The group homomorphism $p: G\rightarrow A$ is surjective $($faithfully flat$)$ iff  $p^*\Rep_k^f(A)$ is a full subcategory of $\Rep_k^f(G)$ and closed under taking subobjects $($subquotients$)$.
    \item[(2)] The group homomorphism $q:L\rightarrow G$ is injective $($a closed immersion$)$ iff any object of $\Rep_k^f(L)$ is a subquotient of an object of the form $q^*(V)$ for some $V\in \Rep_k^f(G)$.
    \item[(3)] Suppose that $p$ is faithfully flat and that $q$ is a closed immersion. Then the sequence $L\xrightarrow{q}G\xrightarrow{p} A$ is exact iff the following conditions are fulfilled:
    \begin{enumerate}
        \item[(a)] For an object $V\in \Rep_k^f(G)$,  $q^* V\in\Rep_k^f(L)$ is trivial iff $V\cong p^*U$ for some $U\in\Rep_k^f(A)$.\label{3a}
        \item[(b)]\label{3b} Let $W_0$ be the maximal trivial subobject of $q^*V$ in $\Rep_k^f(L)$. Then there exists $V_0\hookrightarrow V$ in $\Rep_k^f(G)$ such that $q^*V_0\cong W_0$.
        \item[(c)] Any $W \in \Rep_k^f(L)$ is embeddable in $q^*(V)$ for some $V \in \Rep_k^f(G)$.
    \end{enumerate}
\end{enumerate}
\end{Theorem}

\begin{Definition}
    Let $k$ be a field, $X$ a connected scheme over $k$, $\mathcal{C}_X$ the Tannakian category over $X$. Define the \textit{saturation category} $\overline{\mathcal{C}}_X$ of $\mathcal{C}_X$ as the full subcategory of $\Vect(X)$ whose objects are those $E$ for which there exists a filtration
    $$0=E_0\hookrightarrow E_1 \hookrightarrow \cdots\hookrightarrow E_n=E,$$
    such that $E^i=E_{i+1}/E_i\in\mathcal{C}_X$ for any $0\leq i\leq n-1$.
\end{Definition}

\begin{Remark}
    $\mathcal C_X$ and $\overline{\mathcal{C}}_X$ have the same simple objects.
\end{Remark}

\begin{Proposition}\label{saturatedtannakian}
    Let $k$ be a field, $X$ a connected scheme over $k$, $x\in X(k)$, $\mathcal{C}_X$ the Tannakian category over $X$. Then $\overline{\mathcal{C}}_X$ is a $k$-linear abelian rigid tensor category.
\end{Proposition}

\begin{proof}
For $E,F\in\overline{\mathcal{C}}_X$, choose finite filtrations with subquotients $\operatorname{gr}_i(E)$ and $\operatorname{gr}_j(F)$ in $\mathcal{C}_X$. Then $E^\vee$ admits a filtration with subquotients $\operatorname{gr}_i(E)^\vee\in\mathcal C_X$, and $E\otimes F$ admits a filtration with subquotients $\operatorname{gr}_i(E)\otimes \operatorname{gr}_j(F)\in\mathcal C_X$. It follows that $E^\vee, E\otimes F \in \overline{\mathcal C}_X$. Hence $\overline{\mathcal C}_X$ is a rigid tensor category.

Now we show that $\overline{\mathcal C}_X$ is an abelian category. For a morphism $f\colon E\to F$ in $\overline{{\mathcal C}}_X$, one proves that the kernel and cokernel belong to $\overline{\mathcal C}_X$ by induction on $\operatorname{rk}(E)+\operatorname{rk}(F)$. If $E$ and $F$ are simple objects in $\mathcal{C}_X$, then this claim follows immediately since $\mathcal C_X$ is abelian. If $F$ is nonsimple, take $0\rightarrow F'\xrightarrow{i} F\xrightarrow{j} F''\rightarrow 0$ with $F''$ is simple in $\overline{\mathcal{C}}_X$. Note that $\rk(E)+\rk(F'')< \rk(E)+\rk(F)$, by inductive hypothesis, the kernel and coker of the composite morphism $E\xrightarrow{f} F\xrightarrow{j} F''$ should be in $\overline{\mathcal{C}}_X$. Consider the exact sequence $ 0\rightarrow\Ima(j\circ f)\rightarrow F'' \rightarrow\coker(j\circ f)\rightarrow0$. If $\rk(E)\leq\rk(F'')$, then $\rk(\ker(j\circ f))+\rk(E)  <\rk(E)+\rk(F)$. Applying the inductive hypothesis to the inclusion $\ker(j\circ f)\hookrightarrow E$, we have $\Ima(j\circ f)\cong\coker\bigl(\ker(j\circ f)\hookrightarrow E\bigr)\in\overline{\mathcal C}_X$. If $\rk(E)>\rk(F'')$, then $ \rk(F'')+\rk(\coker(j\circ f))<\rk(E)+\rk(F)$. Applying the inductive hypothesis to the quotient $F''\twoheadrightarrow\coker(j\circ f)$, we obtain $\Ima(j\circ f)\cong\ker\bigl(F''\twoheadrightarrow\coker(j\circ f)\bigr)\in\overline{\mathcal C}_X$.
Thus $\Ima(j\circ f)\in\overline{\mathcal C}_X$. If $\mathrm{Im}(j\circ f)=0$, then $f$ factors through a morphism $g:E\to F'$. Then we have $\ker f\cong\ker g$ and an exact sequence $0\rightarrow \operatorname{coker} g\rightarrow \operatorname{coker} f\rightarrow F''\rightarrow 0$ by the snake lemma. By induction $\ker f$ and $\operatorname{coker} f\in\overline{\mathcal{C}}_X$. If $\mathrm{Im}(j\circ f)=F''$, we have the following commutative diagram of vector bundles
    \[
        \begin{tikzcd}
            0\arrow[r]&\ker(j\circ f)\arrow[r]\arrow[d,"u"]&E\arrow[d,"f"]\arrow[r,"{j\circ f}"]&F''\arrow[d,equal]\arrow[r]&0\\
            0\arrow[r]&F'\arrow[r,"i"]&F\arrow[r,"j"]&F''\arrow[r]&0.
        \end{tikzcd}
    \]
It follows that $\ker f\cong \ker u$ and $\operatorname{coker} f\cong \operatorname{coker}u\in\overline{\mathcal{C}}_X$ by induction. It remains to consider the case where $F$ is simple in $\mathcal C_X$ but $E$ is not. Take $0\rightarrow E'\xrightarrow{p} E\xrightarrow{q} E''\rightarrow 0$ with $E'$ is simple in $\mathcal{C}_X$. If $\ker(f\circ p)=E'$, then $f$ factors through a morphism $h:E''\to F$. Then we have $\operatorname{coker} f\cong\operatorname{coker}h$ and an exact sequence $0\rightarrow E'\rightarrow \ker f\rightarrow \ker h\rightarrow 0$ by the snake lemma. By induction $\ker f$ and $\operatorname{coker} f\in\overline{\mathcal{C}}_X$. If $\ker(f\circ p)=0$, we have the following commutative diagram of vector bundles
    \[
        \begin{tikzcd}
            0\arrow[r]& E'\arrow[r,"p"]\arrow[d,equal]&E\arrow[d,"f"]\arrow[r,"q"]&E''\arrow[d,"v"]\arrow[r]&0\\
            0\arrow[r]&E'\arrow[r,"{f\circ p}"]&F\arrow[r]&\operatorname{coker}(f\circ p)\arrow[r]&0.
        \end{tikzcd}
    \]
It follows that $\ker f\cong \ker v$ and $\operatorname{coker} f\cong \operatorname{coker}v\in\overline{\mathcal{C}}_X$ by induction.
\end{proof}

\begin{Definition}
    Let $k$ be a field, $X$ a connected scheme over $k$, $x\in X(k)$, $\mathcal{C}_X$ a Tannakian category over $X$. The Tannaka group scheme $\pi(\overline{\mathcal{C}}_X,x)$ of $\overline{\mathcal{C}}_X$ is called the \textit{saturation} of $\pi(\mathcal{C}_X,x)$. If $\overline{\mathcal{C}}_X=\mathcal{C}_X$, i.e. $\pi(\overline{\mathcal{C}}_X,x)= \pi(\mathcal{C}_X,x)$, then $\mathcal{C}_X$ and $\pi(\mathcal{C}_X,x)$ are said to be \textit{saturated}.
\end{Definition}

\begin{Proposition}
    Let $k$ be a field, $X$ a connected scheme over $k$, $x\in X(k)$, $\mathcal{C}_X$ a Tannakian categories over $X$. Then the natural homomorphism $\pi(\overline{\mathcal{C}}_X,x)\rightarrow \pi(\mathcal{C}_X,x)$ is faithfully flat.
\end{Proposition}

\begin{proof}
    We have a natural fully faithful functor $f:\mathcal{C}_X\rightarrow \overline{\mathcal{C}}_X$ sending $E\in \mathcal{C}_X$ to $E\in \overline{\mathcal{C}}_X$. For any $E\in\mathcal{C}_X$ and any subobject $F\hookrightarrow E\in\overline{\mathcal{C}}_X$, we claim $F\in\mathcal{C}_X$. Consider the filtration of $F$:
    $$0\hookrightarrow F_1 \hookrightarrow \cdots\hookrightarrow F_n=F,$$
    such that $F_{i+1}/F_i\in\mathcal{C}_X$ for any $i$. Note that $F_1\in\mathcal{C}_X$, so we have $E/F_1\in\mathcal{C}_X$ and the isomorphism $E/F_2\cong (E/F_1)/(F_2/F_1)$ implies that $E/F_2\in\mathcal{C}_X$. Consider the exact sequence 
    $$0\rightarrow F_2\rightarrow E\rightarrow E/F_2\rightarrow 0,$$
    we have that $F_2\in\mathcal{C}_X$ since $\mathcal{C}_X$ is abelian. Repeating this method, we have $F_i\in\mathcal{C}_X$ for any $i$. Hence the functor $f$ is closed under taking subobjects. Therefore, the natural homomorphism $\pi(\overline{\mathcal{C}}_X,x)\rightarrow \pi(\mathcal{C}_X,x)$ is faithfully flat by Theorem~\ref{Thm1}.
\end{proof}

\section{Base change of Tannaka group schemes}

\subsection{Base change of Tannaka group schemes under field extension}

\begin{Lemma}[{\cite[Proposition A.1]{Ota17}}]\label{fullyfaithful}
    Let $\mathcal{C}$ be a $k$-linear abelian rigid tensor category, $\mathcal{D}$ a $k$-linear abelian tensor category, $F : C \rightarrow D$ an exact tensor functor. Suppose there exists a family $A \subseteq Ob(C)$ satisfying the following conditions:
    \begin{enumerate}
        \item The unit object $I$ belongs to $A$.
        \item For any $W_1,W_2 \in A$, the map $F : \Hom_C(W_1, W_2) \rightarrow  \Hom_D(F(W_1), F(W_2))$ is bijective.
        \item For any $V \in Ob(C)$, there exists an $W \in A$ such that $V \hookrightarrow W$.
    \end{enumerate}
Then the functor $F$ is fully faithful.
\end{Lemma}

\begin{Proposition}\label{fieldgeneral}
    Let $k$ be a field, $K/k$ a field extension, $X$ a connected quasi-compact and quasi-separated scheme over $k$, $x_K\in X_K(K)$ lying over $x\in X(k)$, $\mathcal{C}_X$ and $\mathcal{C}_{X_K}$ the Tannakian categories on $X$ and $X_K$ respectively, $\phi:P\rightarrow X$ the universal principal $\pi(\mathcal{C}_X,x)$-bundle.
    \begin{enumerate}
        \item Define the functor
    $$\begin{aligned}
        \eta_K^{P_K}:\Rep_{K}^f(\pi(\mathcal{C}_X,x)_K)&\rightarrow \Qcoh(X_K), V\mapsto((\mathcal{O}_{P}\otimes_k K)\otimes_{K} V)^{\pi(\mathcal{C}_X,x)_K}.
    \end{aligned}$$
    Then $\eta_K^{P_K}$ is fully faithful. Moreover, we have
    \[\eta_K^{P_K}(\Rep_{K}^f(\pi(\mathcal{C}_X,x)_K))=\left\{\mathcal{E}\in \Vect(X_K)\middle| \exists E_1,E_2\in\mathcal{C}_X, \text{ s.t. }E_1\otimes_k K\twoheadrightarrow\mathcal{E}\hookrightarrow E_2\otimes_k K\right\}.\]
        \item If for any $E\in\mathcal{C}_X$, we have $E\otimes_k K \in \mathcal{C}_{X_K}$, then it induces a natural functor
        $$\begin{aligned}
        \eta_K^{P_K}:\Rep_{K}^f(\pi(\mathcal{C}_X,x)_K)&\rightarrow \mathcal{C}_{X_K}, V\mapsto((\mathcal{O}_{P}\otimes_k K)\otimes_{K} V)^{\pi(\mathcal{C}_X,x)_K}.
    \end{aligned}$$
        This functor corresponds to a natural homomorphism of group schemes over $K$
        $$\pi(\mathcal{C}_{X_K},x_K)\rightarrow \pi(\mathcal{C}_X,x)_K.$$
        \item If for any $E\in\mathcal{C}_X$, we have $E\otimes_k K \in \mathcal{C}_{X_K}$, then the following conditions are equivalent:
        \begin{enumerate}
            \item The natural homomorphism $\pi(\mathcal{C}_{X_K},x_K)\rightarrow \pi(\mathcal{C}_X,x)_K$ is faithfully flat.
            \item The functor $\eta_K^{P_K}$ is observable.
            \item $\eta_K^{P_K}(\Rep_{K}^f(\pi(\mathcal{C}_X,x)_K))=\langle \{E\otimes_k K\}_{E\in\mathcal{C}_X}\rangle_{\mathcal{C}_{X_K}}$.
            \item $\eta_K^{P_K}(\Rep_{K}^f(\pi(\mathcal{C}_X,x)_K))=\{\mathcal{E}\in\mathcal{C}_{X_K}|\exists E\in\mathcal{C}_X,\text{ s.t. }\mathcal{E}\hookrightarrow E\otimes_k K\}$.
            \item For any $E\in \mathcal{C}_X$ and any $\mathcal{E}\hookrightarrow E\otimes_k K\in\mathcal{C}_{X_K}$, there exists $E'\in\mathcal{C}_X$ such that $\mathcal{E}^\vee\hookrightarrow E'\otimes_k K$.
        \end{enumerate}
    \end{enumerate}
\end{Proposition}

\begin{proof}
    (1) For the principal $\pi(\mathcal{C}_X,x)$-bundle $\phi:P\rightarrow X$, we have a fully faithful exact tensor functor
     $$\begin{aligned}
        \eta_k^{P}:\Rep_{k}^f(\pi(\mathcal{C}_X,x))&\rightarrow \Qcoh(X),W\mapsto(\mathcal{O}_{P}\otimes_{k} W)^{\pi(\mathcal{C}_X,x)}.
    \end{aligned}$$
    After base change to $K$, we obtain a principal $\pi(\mathcal{C}_X,x)_K$-bundle $\phi_K:P_K\rightarrow X_K$ and an exact tensor functor
    $$\begin{aligned}
        \eta_K^{P_K}:\Rep_{K}^f(\pi(\mathcal{C}_X,x)_K)&\rightarrow \Qcoh(X_K), V\mapsto((\mathcal{O}_{P}\otimes_k K)\otimes_{K} V)^{\pi(\mathcal{C}_X,x)_K}.
    \end{aligned}$$
    For any $W\in\Rep_{k}^f(\pi(\mathcal{C}_X,x))$, we have $W\otimes_k K\in\Rep_{K}^f(\pi(\mathcal{C}_X,x)_K)$ and 
    $$\begin{aligned}
    \eta_K^{P_K}(W\otimes_k K)=((\mathcal{O}_{P}\otimes_k K)\otimes_{K} (W\otimes_k K))^{\pi(\mathcal{C}_X,x)_K}\cong (\mathcal{O}_P\otimes_k W)^{\pi(\mathcal{C}_X,x)}\otimes_k K=\eta_k^{P}(W)\otimes_k K.
    \end{aligned}$$
    For any $W_1,W_2\in\Rep_{k}^f(\pi(\mathcal{C}_X,x))$, by \cite[p.29]{Jan03} and since $X$ is quasi-compact and quasi-separated,
    $$\begin{aligned}
        \Hom_{\pi(\mathcal{C}_X,x)_K}(W_1\otimes_k K,W_2\otimes_k K)&\cong \Hom_{\pi(\mathcal{C}_X,x)}(W_1,W_2)\otimes_k K\\
        &\cong\Hom_{\mathcal{O}_X}(\eta_k^{P}(W_1),\eta_k^{P}(W_2))\otimes_k K\\
        &\cong \Hom_{\mathcal{O}_{X_K}}(\eta_k^{P}(W_1)\otimes_k K,\eta_k^{P}(W_2)\otimes_k K)\\
        &\cong \Hom_{\mathcal{O}_{X_K}}(\eta_K^{P_K}(W_1\otimes_k K),\eta_K^{P_K}(W_2\otimes_k K)).
    \end{aligned}$$
    However, for any $V\in \Rep_{K}^f(\pi(\mathcal{C}_X,x)_K)$, by Lemma~\ref{basechangeembedding} there exists $W\in\Rep_k^f(\pi(\mathcal{C}_X,x))$  such that $V\hookrightarrow W\otimes_k K$ in $\Rep_{K}^f(\pi(\mathcal{C}_X,x)_K)$. Then $\eta_K^{P_K}$ is fully faithful according to Lemma~\ref{fullyfaithful}.

    For any $V\in\Rep_{K}^f(\pi(\mathcal{C}_X,x)_K)$, by Lemma~\ref{basechangeembedding} there exist $W_1,W_2\in\Rep_{k}^f(\pi(\mathcal{C}_X,x))$ such that
    $$W_1\otimes_k K\twoheadrightarrow V \hookrightarrow W_2\otimes_k K$$
    as $\pi(\mathcal{C}_{X},x)_K$-representations. Applying the functor $\eta_K^{P_K}$, we obtain
    $$\eta_k^{P}(W_1)\otimes_k K\twoheadrightarrow \eta_K^{P_K}(V) \hookrightarrow \eta_k^{P}(W_2)\otimes_k K.$$
    
    Conversely, given $\mathcal{E}\in \Qcoh(X_K)$ with
    $$f: E_1\otimes_k K\twoheadrightarrow\mathcal{E}\hookrightarrow E_2\otimes_k K,$$
    for some $E_1,E_2\in\mathcal{C}_X$, there exists $W_1,W_2\in\Rep_k^f(\pi(\mathcal{C}_X,x))$ such that $E_i=\eta_k^{P}(W_i)$ for $i=1,2$. Note that $f\in\Hom_{\mathcal{O}_{X_K}}(\eta_K^{P_K}(W_1\otimes_k K),\eta_K^{P_K}(W_2\otimes_k K))$ and $\mathcal{E}=\Ima(f)$. Since $\eta_K^{P_K}$ is fully faithful, it follows that there exists $\alpha\in\Hom_{\pi(\mathcal{C}_X,x)_K}(W_1\otimes_k K,W_2\otimes_k K)$ such that $f=\eta_K^{P_K}(\alpha)$. So $\mathcal{E}=\eta_K^{P_K}(\Ima(\alpha))$, where $\Ima(\alpha)\in \Rep_K^f(\pi(\mathcal{C}_X,x)_K)$.

    Consequently, we have
    \[\eta_K^{P_K}(\Rep_{K}^f(\pi(\mathcal{C}_X,x)_K))=\left\{\mathcal{E}\in \Qcoh(X_K)\middle| \exists E_1,E_2\in\mathcal{C}_X, \text{ s.t. }E_1\otimes_k K\twoheadrightarrow\mathcal{E}\hookrightarrow E_2\otimes_k K\right\}.\]
    
    (2) For any $V\in\Rep_{K}^f(\pi(\mathcal{C}_X,x)_K)$, by Lemma~\ref{basechangeembedding} there exists $W_1,W_2\in\Rep_{k}^f(\pi(\mathcal{C}_X,x))$ such that
    $$\alpha:W_1\otimes_k K\twoheadrightarrow V\hookrightarrow W_2\otimes_k K,$$
    and $V=\Ima(\alpha)$. Applying the functor $\eta_K^{P_K}$, we obtain
    $$\eta_K^{P_K}(\alpha):\eta_k^{P}(W_1)\otimes_k K\twoheadrightarrow \eta_K^{P_K}(V) \hookrightarrow \eta_k^{P}(W_2)\otimes_k K.$$
    Note that $\eta_k^{P}(W_i)\in\mathcal{C}_X$, then by assumption, $\eta_k^{P}(W_i)\otimes_k K\in \mathcal{C}_{X_K}$ for $i=1,2$. Since $\mathcal{C}_{X_K}$ is abelian, $\eta_K^{P_K}(V)=\Ima(\eta_K^{P_K}(\alpha))\in\mathcal{C}_{X_K}$. Hence $\eta_K^{P_K}$ induces a functor from $\Rep_K^f(\pi(\mathcal{C}_X,x)_K)$ to $\mathcal{C}_{X_K}$ which sends $V\in \Rep_K^f(\pi(\mathcal{C}_X,x)_K)$ to $\eta_K^{P_K}(V)\in\mathcal{C}_{X_K}$ and yields a natural homomorphism $\pi(\mathcal{C}_{X_K},x_K)\rightarrow \pi(\mathcal{C}_X,x)_K$.

    (3) In general, we have the following relationship
    $$\eta_K^{P_K}(\Rep_K^f(\pi(\mathcal{C}_X,x)_K))\subseteq \{\mathcal{E}\in\mathcal{C}_{X_K}|\exists E\in\mathcal{C}_X,\text{ s.t. }\mathcal{E}\hookrightarrow E\otimes_k K\}\subseteq \langle \{E\otimes_k K\}_{E\in\mathcal{C}_X}\rangle_{\mathcal{C}_{X_K}}.$$
    
    $(a)\Rightarrow (c)$ For any $\mathcal{E}\in\langle \{E\otimes_k K\}_{E\in\mathcal{C}_X}\rangle_{\mathcal{C}_{X_K}}$, there exist $E\in\mathcal{C}_X$ and subobjects $\mathcal{E}_1\hookrightarrow \mathcal{E}_2\hookrightarrow E\otimes_k K\in\mathcal{C}_{X_K}$ such that $\mathcal{E}\cong \mathcal{E}_2/\mathcal{E}_1$. Assume $E=\eta_k^{P}(W)$ for some $W\in\Rep_k^f(\pi(\mathcal C_X,x))$, since $\pi(\mathcal{C}_{X_K},x_K)\rightarrow \pi(\mathcal{C}_X,x)_K$ is faithfully flat, the functor $\eta_K^{P_K}$ is closed under taking subobjects by Theorem~\ref{Thm1}, which implies that there exist $V_1\hookrightarrow V_2\hookrightarrow W\otimes_k K\in\Rep_K^f(\pi(\mathcal{C}_X,x)_K)$ such that $\mathcal{E}_i\cong \eta_K^{P_K}(V_i)$. It follows that
    $$\mathcal{E}\cong \mathcal{E}_2/\mathcal{E}_1\cong \eta_K^{P_K}(V_2/V_1)\in \eta_K^{P_K}(\Rep_K^f(\pi(\mathcal{C}_X,x)_K)).$$
    Then $\eta_K^{P_K}(\Rep_K^f(\pi(\mathcal{C}_X,x)_K))= \langle \{E\otimes_k K\}_{E\in\mathcal{C}_X}\rangle_{\mathcal{C}_{X_K}}$.
    
    $(c)\Rightarrow (d)$ According to the above general relationship, we immediately obtain that $$\eta_K^{P_K}(\Rep_K^f(\pi(\mathcal{C}_X,x)_K))= \{\mathcal{E}\in\mathcal{C}_{X_K}|\exists E\in\mathcal{C}_X,\text{ s.t. }\mathcal{E}\hookrightarrow E\otimes_k K\}.$$
    
    $(d)\Rightarrow (e)$ The equality $\eta_K^{P_K}(\Rep_K^f(\pi(\mathcal{C}_X,x)_K))= \{\mathcal{E}\in\mathcal{C}_{X_K}|\exists E\in\mathcal{C}_X,\text{ s.t. }\mathcal{E}\hookrightarrow E\otimes_k K\}$ implies that $\{\mathcal{E}\in\mathcal{C}_{X_K}|\exists E\in\mathcal{C}_X,\text{ s.t. }\mathcal{E}\hookrightarrow E\otimes_k K\}$ is a Tannakian category.
    
    $(e)\Rightarrow (b)$ It follows immediately by Lemma~\ref{Observable}.

    $(b)\Rightarrow (a)$ It follows immediately by (1) and Lemma~\ref{Observablesurjective}.
\end{proof}

\begin{Theorem}\label{fieldgeneralcriteria}
    Let $k$ be a field, $K/k$ a field extension, $X$ a connected quasi-compact and quasi-separated scheme over $k$, $x_K\in X_K(K)$ lying over $x\in X(k)$, $p:X_K\rightarrow X$ the natural projection, $\mathcal{C}_X$ and $\mathcal{C}_{X_K}$ the Tannakian categories on $X$ and $X_K$ respectively. Suppose for any $E\in\mathcal C_X$, $E\otimes_k K\in\mathcal C_{X_K}$. Then the following conditions are equivalent:
        \begin{enumerate}
            \item The natural homomorphism $\pi(\mathcal{C}_{X_K},x_K)\rightarrow \pi(\mathcal{C}_X,x)_K$ is an isomorphism.
            \item For any $\mathcal{E}\in\mathcal{C}_{X_K}$, there exists $ E\in\mathcal{C}_X$ such that $\mathcal{E}\hookrightarrow E\otimes_k K$.
            \item For any $\mathcal E\in\mathcal C_{X_K}$, there is a filtered colimit system $\{\mathcal F_\lambda\}_{\lambda\in \Lambda}$ of subsheaves of $ p_*\mathcal E$ such that $\mathcal F_\lambda\in\mathcal C_X$ for every $\lambda$ and $p_*\mathcal E=\varinjlim_\lambda\mathcal F_\lambda$ as $\mathcal O_X$-modules.
        \end{enumerate}
\end{Theorem}

\begin{proof}
    
    In general, we have the following relationship
    $$\eta_K^{P_K}(\Rep_K^f(\pi(\mathcal{C}_X,x)_K))\subseteq \{\mathcal{E}\in\mathcal{C}_{X_K}|\exists E\in\mathcal{C}_X,\text{ s.t. }\mathcal{E}\hookrightarrow E\otimes_k K\}\subseteq \langle \{E\otimes_k K\}_{E\in\mathcal{C}_X}\rangle_{\mathcal{C}_{X_K}}\subseteq \mathcal{C}_{X_K}.$$
    
    $(1)\Rightarrow (2)$ Since $\pi(\mathcal{C}_{X_K},x_K)\xrightarrow{\cong} \pi(\mathcal{C}_X,x)_K$, we have $\mathcal{C}_{X_K}=\eta_K^{P_K}(\Rep_K^f(\pi(\mathcal{C}_X,x)_K))$. Then it follows by the above relationship.
    
    $(2)\Rightarrow(3)$ For any $\mathcal E\in\mathcal C_{X_K}$, there exists $E_1\in\mathcal{C}_X$ such that $\mathcal{E}\hookrightarrow E_1\otimes_k K$. Similarly, there exists $E_2^\vee\in\mathcal{C}_X$ such that $\mathcal{E}^{\vee}\hookrightarrow E_2^\vee\otimes_k K$. It follows that we have a morphism $f:E_2\otimes_k K\twoheadrightarrow\mathcal{E}\hookrightarrow E_1\otimes_k K$. Note that $\Hom_{\mathcal O_{X_K}}(p^*E_2,p^*E_1) \cong \Hom_{\mathcal O_X}(E_2,E_1)\otimes_k K$, so we have $f=\sum_{i=1}^{m}f_i\otimes a_i$, where $f_i\in \Hom_{\mathcal O_X}(E_2,E_1)$ and $a_i\in K$. Since $p$ is affine, applying the exact functor $p_*$ gives a morphism $p_*f:p_*p^*E_2\twoheadrightarrow p_*\mathcal E \hookrightarrow p_*p^*E_1$ and $p_*\mathcal E=\Ima(p_*f)$. Let $\Lambda$ be the set of finite-dimensional $k$-subspaces of $K$, ordered by inclusion. For any $W\in\Lambda$, consider the morphism
\[
    f_W:E_2\otimes_k W \rightarrow E_1\otimes_k\left(\sum_{i=1}^{m}a_iW\right), e\otimes w\mapsto \sum_{i=1}^{m}f_i(e)\otimes a_iw.
\]
    Then $f_W$ is a morphism in $\mathcal C_X$, which implies that $\mathcal F_W:=\Ima(f_W)\in\mathcal C_X$. Note that $p_*p^*E_i\cong E_i\otimes_k  K\cong \varinjlim_{W\in\Lambda} E_i\otimes W$, $i=1,2$. Hence $\Ima(p_*f|_{E_2\otimes_k W})=\Ima(f_W)=\mathcal F_W$. Therefore, each $\mathcal F_W$ is a subsheaf of $p_*\mathcal E$. Finally, $E_2\otimes_k K = \varinjlim_{W\in\Lambda}(E_2\otimes_k W)$, so we have $p_*\mathcal E=\Ima(p_*f)=\varinjlim_{W\in\Lambda}\Ima(p_*f|_{E_2\otimes_k W})=\varinjlim_{W\in\Lambda}\mathcal F_W$.

    $(3)\Rightarrow(1)$ For any $\mathcal E\in\mathcal C_{X_K}$, we have $p_*\mathcal E^\vee =\varinjlim_\lambda\mathcal F_\lambda$, where $\mathcal F_\lambda\in\mathcal C_X$. Since $p$ is affine, we have a natural surjection $p^*p_*\mathcal E^\vee \twoheadrightarrow\mathcal E^\vee$. Pullback commutes with filtered colimits, so the images of the induced morphisms $\mathcal F_\lambda\otimes_k K \rightarrow\mathcal E^\vee$ form a directed family whose union is $\mathcal E^\vee$. Since $X_K$ is quasi-compact and $\mathcal E^\vee$ is of finite type, there exists some $\lambda_0$ such that the morphism $\mathcal F_{\lambda_0}\otimes_k K \rightarrow\mathcal E^\vee$ is surjective. Taking duality, we obtain $\mathcal E\hookrightarrow \mathcal F_{\lambda_0}^\vee\otimes_k K$. Similarly, $\mathcal{E}^\vee \hookrightarrow \mathcal F'\otimes_k K$ for some $\mathcal F'\in\mathcal C_{X}$. It follows that $\mathcal {F'}^\vee\otimes_k K\twoheadrightarrow \mathcal E\hookrightarrow \mathcal F_{\lambda_0}^\vee\otimes_k K$, i.e. $\mathcal{E}\in \eta_K^{P_K}(\Rep_K^f(\pi(\mathcal{C}_X,x)_K))$. Then $\mathcal{C}_{X_K}=\eta_K^{P_K}(\Rep_K^f(\pi(\mathcal{C}_X,x)_K))$. Hence $\pi(\mathcal{C}_{X_K},x_K)\xrightarrow{\cong} \pi(\mathcal{C}_X,x)_K$.
\end{proof}

\begin{Proposition}\label{transurjective}
    Let $k$ be a field, $K/k$ a field extension, $X$ a connected quasi-compact and quasi-separated scheme over $k$, $x_K\in X_K(K)$ lying over $x\in X(k)$, $\mathcal{C}_X\subseteq\mathcal C'_X$ and $\mathcal{C}_{X_K}\subseteq\mathcal C'_{X_K}$ the Tannakian categories on $X$ and $X_K$ respectively such that for any $E\in\mathcal{C}_X$, $E\otimes_k K\in\mathcal{C}_{X_K}$ and for any $E'\in\mathcal{C}'_X$, $E'\otimes_k K\in\mathcal{C}'_{X_K}$. Suppose the natural homomorphism $\pi(\mathcal{C}'_{X_K},x_K)\twoheadrightarrow\pi(\mathcal{C}'_X,x)_K$ is faithfully flat, then the natural homomorphism $\pi(\mathcal{C}_{X_K},x_K)\rightarrow \pi(\mathcal{C}_X,x)_K$ is faithfully flat.
\end{Proposition}

\begin{proof}
    Since for any $E\in\mathcal{C}_X$, $E\otimes_k K\in\mathcal{C}_{X_K}$ and for any $E'\in\mathcal{C}'_X$, $E'\otimes_k K\in\mathcal{C}'_{X_K}$, by [Proposition~\ref{fieldgeneral}, (2)], we have natural homomorphisms $$\pi(\mathcal{C}'_{X_K},x_K)\twoheadrightarrow\pi(\mathcal{C}'_X,x)_K~\text{and}~\pi(\mathcal{C}_{X_K},x_K)\rightarrow \pi(\mathcal{C}_X,x)_K.$$ Then consider the following commutative diagram
    \[
        \begin{tikzcd}
            \pi(\mathcal{C}'_{X_K},x_K)\arrow[r,two heads] \arrow[d,two heads]&\pi(\mathcal{C}'_X,x)_K\arrow[d,two heads]\\
            \pi(\mathcal{C}_{X_K},x_K)\arrow[r]&\pi(\mathcal{C}_X,x)_K
        \end{tikzcd}
    \]
    It follows immediately that the bottom is faithfully flat.
\end{proof}

\begin{Proposition}\label{equivsurjective}
    Let $k$ be a field, $K/k$ a field extension, $X$ a connected quasi-compact and quasi-separated scheme over $k$, $x_K\in X_K(K)$ lying over $x\in X(k)$, $\mathcal{C}_X$ and $\mathcal{C}_{X_K}$ the Tannakian categories on $X$ and $X_K$ respectively. Suppose for any $E\in\mathcal{C}_X$, $E\otimes_k K\in\mathcal{C}_{X_K}$. Then for any $F\in\overline{\mathcal{C}}_X$, $F\otimes_k K\in\overline{\mathcal{C}}_{X_K}$, and the following conditions are equivalent:
    \begin{enumerate}
        \item The natural homomorphism $\pi(\mathcal{C}_{X_K},x_K)\rightarrow \pi(\mathcal{C}_X,x)_K$ is faithfully flat.
        \item The natural homomorphism $\pi(\overline{\mathcal{C}}_{X_K},x_K)\rightarrow \pi(\overline{\mathcal{C}}_X,x)_K$ is faithfully flat.
    \end{enumerate}
\end{Proposition}

\begin{proof}
    Let $F\in \overline{\mathcal{C}}_X$, consider the filtration of $F$:
    $$0 \hookrightarrow F_1\hookrightarrow \cdots\hookrightarrow F_n= F,$$
    such that $F^i=F_{i+1}/F_{i}\in\mathcal{C}_X$ for any $i$. Tensoring with $K$, we obtain a filtration of $F\otimes_k K$: 
    $$0 \hookrightarrow F_1\otimes_k K\hookrightarrow \cdots\hookrightarrow F_n\otimes_k K= F\otimes_k K,$$
    where $(F_{i+1}\otimes_k K)/(F_{i}\otimes_k K)\cong (F_{i+1}/F_i)\otimes_k K\in\mathcal{C}_{X_K}$ for any $i$ by assumption. Then $F\otimes_k K\in\overline{\mathcal{C}}_{X_K}$.

    $(1)\Rightarrow (2)$ Let $E\in \overline{\mathcal{C}}_{X}$, $\mathcal{E}\in\overline{\mathcal{C}}_{X_K}$ of rank $r$ and $\mathcal{E}\hookrightarrow E\otimes_k K\in\overline{\mathcal{C}}_{X_K}$. Note that $\wedge ^{r-1}\mathcal{E}\hookrightarrow \wedge ^{r-1}E\otimes_k K$ and $\det \mathcal{E}\cong \wedge ^{r}\mathcal{E}\hookrightarrow \wedge ^{r}E\otimes_k K$. Since $\det \mathcal{E}$ is of rank $1$, so simple in $\mathcal{C}_{X_K}$, by [Proposition~\ref{fieldgeneral}, (3)], there exists $E'\in\mathcal{C}_X$ such that $(\det \mathcal{E})^\vee\hookrightarrow E'\otimes_k K$. Then 
    $$\mathcal{E}^\vee\cong \wedge ^{r-1}\mathcal{E}\otimes_{\mathcal{O}_{X_K}} (\det \mathcal{E})^\vee\hookrightarrow ((\wedge ^{r-1}E)\otimes_{\mathcal{O}_X}E')\otimes_k K.$$
    Hence the natural homomorphism $\pi(\overline{\mathcal{C}}_{X_K},x_K)\rightarrow \pi(\overline{\mathcal{C}}_X,x)_K$ is faithfully flat by [Proposition~\ref{fieldgeneral}, (3)].
    
    $(2)\Rightarrow (1)$ If the natural homomorphism $\pi(\overline{\mathcal{C}}_{X_K},x_K)\rightarrow \pi(\overline{\mathcal{C}}_X,x)_K$ is faithfully flat, it follows immediately by Proposition~\ref{transurjective}. 
\end{proof}

\begin{Lemma}\label{subquotient}
    Let $k$ be a field, $K/k$ a field extension, $X$ a connected quasi-compact and quasi-separated scheme over $k$, $\mathcal{C}_X$ and $\mathcal{C}_{X_K}$ the saturated Tannakian categories on $X$ and $X_K$ respectively, $E_1,E_2\in\mathcal{C}_X$, $0\rightarrow E_1\otimes_k K\rightarrow \mathcal{E}\rightarrow E_2\otimes_k K\rightarrow 0$ an extension. Suppose for any $E\in\mathcal{C}_X$, $E\otimes_k K\in\mathcal{C}_{X_K}$. Then there exists $F\in \mathcal{C}_X$ such that $\mathcal{E}$ is a subquotient of $F\otimes_k K$ in $\mathcal{C}_{X_K}$.
\end{Lemma}

\begin{proof}
    Let $\sigma$ denote the corresponding element of the extension $$0\rightarrow E_1\otimes_k K\rightarrow \mathcal{E}\rightarrow E_2\otimes_k K\rightarrow 0$$ in $\Ext^1_{\mathcal{O}_{X_K}}(E_2\otimes_k K,E_1\otimes_k K)$. Note that
    $$\Ext^1_{\mathcal{O}_{X_K}}(E_2\otimes_k K,E_1\otimes_k K)\cong \Ext^1_{\mathcal{O}_X}(E_2,E_1)\otimes_k K,$$
    so there exist $\lambda_i\in K^*$ and $\sigma_i\in \Ext^1_{\mathcal{O}_X}(E_2,E_1)$ such that $\sigma\cong \sum\limits_{i=1}^n \sigma_i\otimes_k \lambda_i$. For any $\sigma_i\in \Ext^1_{\mathcal{O}_X}(E_2,E_1)$:
    $$\sigma_i:0\rightarrow E_1\rightarrow F_i\rightarrow E_2\rightarrow 0,$$
    then $F_i\in\mathcal{C}_X$ since $\mathcal{C}_X$ is saturated. Tensoring with $K$, we obtain a collection of extensions 
    $$\sigma_i\otimes_k 1:0\rightarrow E_1\otimes_k K\rightarrow F_i\otimes_k K\rightarrow E_2\otimes_k K\rightarrow 0.$$
    Consider the following morphisms for each $i$:
    $$\begin{aligned}
        \lambda_i: E_1\otimes_k K\rightarrow E_1\otimes_k K,e_1\mapsto \lambda_i e_1.
    \end{aligned}$$
    Then for each $i$, we have the following pushout diagram
    \[
        \begin{tikzcd}
            0\arrow[r]&E_1\otimes_k K\arrow[r]\arrow[d,"\lambda_i"]&F_i\otimes_k K\arrow[r]\arrow[d,"\cong"]&E_2\otimes_k K\arrow[r]\arrow[d,equal]&0\\
            0\arrow[r]&E_1\otimes_k K\arrow[r]&\arrow[ul, phantom, "\lrcorner", very near start]\mathcal{F}_i\arrow[r]&E_2\otimes_k K\arrow[r]&0
        \end{tikzcd}
    \]
    so the bottom extension corresponds to $\sigma_i\otimes_k \lambda_i\in\Ext^1_{\mathcal{O}_{X_K}}(E_2\otimes_k K,E_1\otimes_k K)$. Consider
    $$i: E_2\otimes_k K\rightarrow (E_2\otimes_k K)^{\oplus n},e_2\mapsto (e_2,\cdots,e_2),$$
    $$p: (E_1\otimes_k K)^{\oplus n}\rightarrow E_1\otimes_k K,(e_1^1,\cdots, e_1^n)\mapsto \sum_{i=1}^{n} e_1^i.$$
    They yield the following commutative diagram
    \[\begin{tikzcd}
        0\arrow[r]&E_1^{\oplus n}\otimes_k K\arrow[r]\arrow[d,"p"]&\oplus_{i=1}^n \mathcal{F}_i \arrow[r]\arrow[d, two heads]&E_2^{\oplus n}\otimes_k K\arrow[r]\arrow[d,equal]&0\\
        0\arrow[r]& E_1\otimes_k K\arrow[r]& U_K\arrow[ul, phantom, "\lrcorner", very near start] \arrow[r]& E_2^{\oplus n}\otimes_k K\arrow[r]&0\\
        0\arrow[r]& E_1\otimes_k K\arrow[r]\arrow[u,equal]& P_K\arrow[ur, phantom, "\llcorner", very near start]\arrow[r]\arrow[u, hook]&E_2\otimes_k K\arrow[r]\arrow[u,"i"]&0
    \end{tikzcd}\]
    where $U_K$ is the pushout, $P_K$ is the pullback. The extension $$0\rightarrow E_1\otimes_k K\rightarrow P_K\rightarrow E_2\otimes_k K\rightarrow 0$$ corresponds to $\sigma=\sum\limits_{i=1}^n \sigma_i\otimes_k \lambda_i$ in $\Ext^1_{\mathcal{O}_{X_K}}(E_2\otimes_k K,E_1\otimes_k K)$ by definition of Baer sum, so we have $\mathcal{E}\cong P_K$. Since $\mathcal{C}_X$ and $\mathcal{C}_{X_K}$ are saturated, we have $P_K,U_K\in \mathcal{C}_{X_K}$. Since $\mathcal{E}\cong P_K$ is a subobject of $U_K$ and $U_K$ is a quotient object of $\oplus_{i=1}^n \mathcal{F}_i\cong \oplus_{i=1}^n F_i\otimes_k K$, we have $\mathcal{E}$ is a subquotient of $E\otimes_k K$ for some $E\in\overline{\mathcal{C}}_{X}$.
\end{proof}

\begin{Proposition}\label{saturationisomorphism}
    Let $k$ be a field, $K/k$ a field extension, $X$ a connected quasi-compact and quasi-separated scheme over $k$, $x_K\in X_K(K)$ lying over $x\in X(k)$, $\mathcal{C}_X$ and $\mathcal{C}_{X_K}$ the Tannakian categories on $X$ and $X_K$ respectively. If for any $E\in\mathcal{C}_X$, $E\otimes_k K\in\mathcal{C}_{X_K}$ and the natural homomorphism $\pi(\mathcal{C}_{X_K},x_K)\rightarrow \pi(\mathcal{C}_X,x)_K$ is an isomorphism, then the natural homomorphism $\pi(\overline{\mathcal{C}}_{X_K},x_K)\rightarrow \pi(\overline{\mathcal{C}}_X,x)_K$ is an isomorphism.
\end{Proposition}

\begin{proof}
    For any $\mathcal{E}\in\overline{\mathcal{C}}_{X_K}$, consider the Jordan H\"older filtration of $\mathcal{E}$ in $\overline{\mathcal{C}}_{X_K}$:
    $$0\hookrightarrow \mathcal{E}_1\hookrightarrow\cdots \hookrightarrow \mathcal{E}_n=\mathcal{E},$$
    where $\mathcal{E}_{i+1}/\mathcal{E}_i\in\mathcal{C}_{X_K}$ is simple for any $i$. Then we have an exact sequence
    $$0\rightarrow \mathcal{E}_{1}\rightarrow \mathcal{E}_2\rightarrow \mathcal{E}_2/\mathcal{E}_1\rightarrow 0,$$
    where $\mathcal{E}_1,\mathcal{E}_2/\mathcal{E}_1\in \mathcal{C}_{X_K}$. By Theorem~\ref{fieldgeneralcriteria}, there exist $E_1,E'_2\in\mathcal{C}_X$ such that $\mathcal{E}_1\hookrightarrow E_1\otimes_k K$ and $E'_2\otimes_k K\twoheadrightarrow\mathcal{E}_2/\mathcal{E}_1$. Consider the following commutative diagram
    \[
        \begin{tikzcd}
            0\arrow[r]&\mathcal{E}_1\arrow[r]\arrow[d,hook]&\mathcal{E}_2 \arrow[r]\arrow[d,hook]&\mathcal{E}_2/\mathcal{E}_1\arrow[r]\arrow[d,equal]&0\\
        0\arrow[r]& E_1\otimes_k K\arrow[r]& \mathcal{E}'_2\arrow[ul, phantom, "\lrcorner", very near start] \arrow[r]& \mathcal{E}_2/\mathcal{E}_1\arrow[r]&0\\
        0\arrow[r]& E_1\otimes_k K\arrow[r]\arrow[u,equal]& \mathcal{E}''_2\arrow[ur, phantom, "\llcorner", very near start]\arrow[r]\arrow[u,two heads]&E'_2\otimes_k K\arrow[r]\arrow[u,two heads]&0
        \end{tikzcd}
    \]
    It follows that $\mathcal{E}_2$ is a subquotient of $\mathcal{E}''_2$ in $\overline{\mathcal{C}}_{X_K}$ since $\overline{\mathcal{C}}_{X_K}$ is saturated. Applying Lemma~\ref{subquotient}, $\mathcal{E}''_2$ is a subquotient of $E\otimes_k K$ for some $E\in\overline{\mathcal{C}}_X$, so $\mathcal{E}_2$ is a subquotient of $E\otimes_k K$ in $\overline{\mathcal{C}}_{X_K}$. By Proposition~\ref{equivsurjective}, the natural homomorphism $$\pi(\overline{\mathcal{C}}_{X_K},x_K)\rightarrow \pi(\overline{\mathcal{C}}_X,x)_K$$ is faithfully flat, so [Proposition~\ref{fieldgeneral}, (3)] implies that $\mathcal{E}_2\hookrightarrow E_2\otimes_k K$ for some $E_2\in\overline{\mathcal{C}}_X$. Repeating this method, for any $i$, there exists $E_i\in\overline{\mathcal{C}}_X$ such that $\mathcal{E}_i\hookrightarrow E_i\otimes_k K$ in $\overline{\mathcal{C}}_{X_K}$. In particular, $\mathcal{E}\hookrightarrow E\otimes_k K$ for some $E\in\overline{\mathcal{C}}_X$. Hence the natural homomorphism $$\pi(\overline{\mathcal{C}}_{X_K},x_K)\rightarrow \pi(\overline{\mathcal{C}}_X,x)_K$$ is an isomorphism by Theorem~\ref{fieldgeneralcriteria}.
\end{proof}

\subsection{Base change of Tannaka group schemes under finite extension}

\begin{Theorem}\label{fieldfinite}
Let $k$ be a field, $K/k$ a finite field extension, $X$ a connected quasi-compact and quasi-separated scheme over $k$, $x_K\in X_K(K)$ lying over $x\in X(k)$, $p:X_K\rightarrow X$ the projection, $\mathcal{C}_X$ and $\mathcal{C}_{X_K}$ the Tannakian categories on $X$ and $X_K$ respectively. Suppose for any $E\in\mathcal{C}_X$, $E\otimes_k K\in\mathcal{C}_{X_K}$. Then the following conditions are equivalent:
\begin{enumerate}
  \item For any $\mathcal E\in\mathcal C_{X_K}$, we have $p_*\mathcal E\in\mathcal C_X$.
  \item The natural homomorphism $\pi(\mathcal C_{X_K},x_K)\rightarrow\pi(\mathcal C_X,x)_K$ is an isomorphism.
  \item \begin{enumerate}
            \item For any $E\in\Vect(X)$, $E\in\mathcal C_X$ iff $E\otimes_k K\in\mathcal C_{X_K}$.
            \item For any $\mathcal E\in\mathcal C_{X_K}$, we have $p^*p_*\mathcal E\in\mathcal C_{X_K}$.
  \end{enumerate}
\end{enumerate}
\end{Theorem}

\begin{proof}
$(1)\Rightarrow(2)$ It follows immediately by Theorem~\ref{fieldgeneralcriteria}.

$(2)\Rightarrow(1)$ For any $\mathcal E\in\mathcal C_{X_K}$, by Theorem~\ref{fieldgeneralcriteria}, there is a filtered colimit system $\{\mathcal F_\lambda\}_{\lambda\in \Lambda}$ of subsheaves of $ p_*\mathcal E$ such that $\mathcal F_\lambda\in\mathcal C_X$ for every $\lambda$ and $p_*\mathcal E=\varinjlim_\lambda\mathcal F_\lambda$ as $\mathcal O_X$-modules. The fibres $\{{\mathcal F_{\lambda}}|_{x}\}_{\lambda\in \Lambda}$ form a filtered colimit system of $(p_*\mathcal E)|_{x}$. Since $(p_*\mathcal E)|_{x}$ is finite-dimensional over $k$, there exists $\lambda_0$ such that $(p_*\mathcal E)|_{x}\cong {\mathcal F_{\lambda_0}}|_{x}$. Suppose there exists $\mu\in\Lambda$ such that $\mathcal F_{\lambda_0}\subseteq \mathcal F_{\mu}$, then we have an exact sequence 
$$0\rightarrow \mathcal F_{\lambda_0}\rightarrow \mathcal F_{\mu}\rightarrow \mathcal F_{\mu}/\mathcal F_{\lambda_0}\rightarrow 0.$$
Taking fibre, we obtain $(\mathcal F_{\mu}/\mathcal F_{\lambda_0})|_{x}=0$, which implies that $\mathcal F_{\mu}/\mathcal F_{\lambda_0}=0$, i.e. $\mathcal F_{\mu}=\mathcal F_{\lambda_0}$. Hence $p_*\mathcal{E}\cong \mathcal F_{\lambda_0}\in\mathcal C_X$.

$(1)+(2)\Rightarrow(3)$ (a) For any $E\in\Vect(X)$, if $E\otimes_k K\in\mathcal{C}_{X_K}$, by the isomorphism $\pi(\mathcal{C}_{X_K},x_K)\cong \pi(\mathcal{C}_X,x)_K$ and Theorem~\ref{fieldgeneralcriteria}, there exist $E_1,E_2\in\mathcal{C}_X$ such that
    $$E_1\otimes_k K\twoheadrightarrow E\otimes_k K\hookrightarrow E_2\otimes_k K.$$
    For any $F\in\Vect(X)$, by projection formula, we have
    $$p_*(F\otimes_k K)\cong F \otimes_{\mathcal{O}_X}(p_* \mathcal{O}_{X_K})\cong F \otimes_{\mathcal{O}_X}(\mathcal{O}_{X}\otimes_k K)\cong F \otimes_{\mathcal{O}_X}\mathcal{O}_{X}^{\oplus [K:k]}\cong \oplus_{[K:k]} F.$$
    Applying the exact functor $p_*$ to the above morphisms, we obtain
    $$\oplus_{[K:k]} E_1\in\mathcal{C}_X\twoheadrightarrow \oplus_{[K:k]} E\hookrightarrow \oplus_{[K:k]} E_2\in\mathcal{C}_X.$$
    It follows that $\oplus_{[K:k]} E\in\mathcal{C}_X$. Hence $E\in\mathcal{C}_X$.
    
    (b) It follows by the condition (1).

$(3)\Rightarrow (1)$ Obviously.
\end{proof}

\begin{Lemma}\label{fieldfinitegalois}
    Let $k$ be a field, $K/k$ a finite Galois field extension, $X$ a connected quasi-compact and quasi-separated scheme over $k$, $x_K\in X_K(K)$ lying over $x\in X(k)$, $p:X_K\rightarrow X$ the projection, $\mathcal{C}_X$ and $\mathcal{C}_{X_K}$ the Tannakian categories on $X$ and $X_K$ respectively. Then the following conditions are equivalent:
    \begin{enumerate}
        \item For any $\mathcal E\in\mathcal C_{X_K}$, we have $p^*p_*\mathcal E\in\mathcal C_{X_K}$.
        \item For any $\mathcal{E}\in\mathcal{C}_{X_K}$ and any automorphism $\tilde{g}:X_K\rightarrow X_K$ induced by $g\in\Gal(K/k)$, $\tilde{g}^*\mathcal{E}\in\mathcal{C}_{X_K}$.
    \end{enumerate}
\end{Lemma}

\begin{proof}
    Consider the following diagram
    \[
        \begin{tikzcd}
            (X_K)\times \Gal(K/k)\arrow[r,"\cong"]\arrow[rr,bend left,"p_1"]\arrow[rd,"\mu"]& (X_K)\times_X (X_K)\arrow[r,"P_2"]\arrow[d,"P_1"]& X_K\arrow[d,"p"]\arrow[r]&\Spec K\arrow[d]\\
            &X_K\arrow[r,"p"]& X\arrow[r]&\Spec k
        \end{tikzcd}
    \]
    where $\mu$ is the action of $\Gal(K/k)$ on $X_K$. By cohomology and flat base change, for any $\mathcal{E}\in\mathcal{C}_{X_K}$, we have $$p^*p_*\mathcal{E}\cong {P_2}_*P_1^*\mathcal{E}\cong {p_1}_*\mu^* \mathcal{E}=\bigoplus_{g\in\Gal(K/k)}\tilde{g}^*\mathcal{E}.$$
    
    $(1)\Rightarrow (2)$ Consider the morphisms $i:\tilde{g}^*\mathcal E\rightarrow p^*p_*\mathcal E$ and $q:p^*p_*\mathcal E\rightarrow \tilde{g}^*\mathcal E$ induced by the direct sum composition. Note that the image of $i\circ q:p^*p_*\mathcal E\rightarrow p^*p_*\mathcal E$ is exactly $\tilde{g}^*\mathcal E$, which implies that $\tilde{g}^*\mathcal E\in\mathcal C_{X_K}$.

    $(2)\Rightarrow (1)$ We have $p^*p_*\mathcal{E}\cong \bigoplus_{g\in\Gal(K/k)}\tilde{g}^*\mathcal{E}$, which implies that $p^*p_*\mathcal{E}\in\mathcal{C}_{X_K}$.
\end{proof}

\begin{Lemma}\label{inseparable-filtration}
Let $k$ be a field of characteristic $p>0$, $K/k$ a finite purely inseparable extension, $X$ a connected quasi-compact and quasi-separated scheme over $k$, $q:X_K\rightarrow X$ the projection, $\mathcal E\in\Vect(X_K)$. We have the following Cartesian diagram:
\[
    \begin{tikzcd}
        X_K\times_X X_K\arrow[r,"q_2"]\arrow[d,"q_1"]&X_K\arrow[d,"q"]\\
        X_K\arrow[r,"q"]&X
    \end{tikzcd}
\]
Then the vector bundle ${q}^*q_*\mathcal{E}$ has a finite filtration by vector subbundles whose successive quotients are finite direct sums of copies of $\mathcal E$.
\end{Lemma}

\begin{proof}
Choose generators $a_1,\ldots,a_r$ of $K$ over $k$. The ideal $I=\ker\bigl(K\otimes_k K\rightarrow K,\ a\otimes b\mapsto ab\bigr)$ is generated by $1\otimes a_i-a_i\otimes1$. Since $K/k$ is purely inseparable, there is an integer $e$ such that $a_i^{p^e}\in k$ for every $i$. Consequently, $(1\otimes a_i-a_i\otimes1)^{p^e}=0$. It follows that $I$ is nilpotent. Moreover, $(K\otimes_k K)/I\cong K$, so that $K\otimes_k K$ is a local Artinian $K$-algebra with residue field $K$.

Note that $\mathcal I=\tilde{I}\mathcal O_{X_K\times_X X_K}$ is sheaf of the ideal of the diagonal $\Delta:X_K\hookrightarrow {X_K\times_X X_K}$. The morphism $q_1$ identifies ${X_K\times_X X_K}$ with $X_K\times_K\Spec (K\otimes_k K)$. Since $q_2^*\mathcal E$ is a vector bundle, we have
\[
  \mathcal I^j q_2^*\mathcal E/\mathcal I^{j+1}q_2^*\mathcal E \cong (\mathcal I^j/\mathcal I^{j+1}) \otimes_{\mathcal O_{X_K\times_X X_K}}q_2^*\mathcal E.
\]
Since $\mathcal I\cdot \bigl(\mathcal I^j/\mathcal I^{j+1}\bigr)=0$, we obtain  $\mathcal I((\mathcal I^j/\mathcal I^{j+1})\otimes_{\mathcal O_{X_K\times_X X_K}}q_2^*\mathcal E)=0$. It follows that 
$$\Delta_*\Delta^*((\mathcal I^j/\mathcal I^{j+1})\otimes_{\mathcal O_{X_K\times_X X_K}}q_2^*\mathcal E)\cong (\mathcal I^j/\mathcal I^{j+1})\otimes_{\mathcal O_{X_K\times_X X_K}}q_2^*\mathcal E.$$ 
By the isomorphism $X_K\times_X X_K\cong X_K\times_K\Spec (K\otimes_k K)$
and the exactness of tensoring over $K$, we have
$$\Delta^*(\mathcal I^j/\mathcal I^{j+1})\cong\mathcal O_{X_K}\otimes_K(I^j/I^{j+1}).$$
Moreover, $q_2\circ\Delta=\mathrm{id}_{X_K}$, so that $\Delta^*q_2^*\mathcal E \cong\mathcal E$. Consequently,
$$\mathcal I^j q_2^*\mathcal E/\mathcal I^{j+1}q_2^*\mathcal E  \cong \Delta_*\!\left(\Delta^*(\mathcal I^j/\mathcal I^{j+1}) \otimes_{\mathcal O_{X_K}}\Delta^*q_2^*\mathcal E\right)\cong \Delta_*\!\left(\mathcal E\otimes_K(I^j/I^{j+1})\right).$$
Since $q_1\circ\Delta=\mathrm{id}_{X_K}$, applying $q_{1*}$ gives
$$ q_{1*}\!\left( \mathcal I^j q_2^*\mathcal E/\mathcal I^{j+1}q_2^*\mathcal E\right) \cong (q_1\circ\Delta)_*\!\left(\mathcal E\otimes_K(I^j/I^{j+1})\right)\cong\mathcal E\otimes_K(I^j/I^{j+1}).$$
Define $J^j(\mathcal E)=q_{1*}(\mathcal I^j q_2^*\mathcal E)$, then we have a filtration of vector bundles
\[
  q^*q_*\mathcal E\cong{q_1}_*q_2^*\mathcal{E}=J^0(\mathcal E) \hookleftarrow J^1(\mathcal E)\hookleftarrow\cdots \hookleftarrow J^N(\mathcal E)=0.
\]
Since $q_1$ is finite, its pushforward is exact, and we obtain 
$$J^j(\mathcal E)/J^{j+1}(\mathcal E) \cong \mathcal E\otimes_K(I^j/I^{j+1}) \cong \mathcal E^{\oplus m_j},$$
where $m_j=\dim_K(I^j/I^{j+1})$. Finally, we have $\sum_{j=0}^{N-1}m_j=\dim_K(K\otimes_k K)=[K:k]$.
\end{proof}

\begin{Lemma}\label{fieldfiniteinseparable}
Let $k$ be a field of characteristic $p>0$, $K/k$ a finite purely inseparable field extension of degree $p^n$, $X$ a connected quasi-compact and quasi-separated scheme over $k$, $x_K\in X_K(K)$ lying over $x\in X(k)$, $q:X_K\rightarrow X$ the projection, $\mathcal{V}_{X_K}$ a subcategory of $\Vect(X_K)$ which is closed under taking extensions. Then for any $\mathcal E\in\mathcal{V}_{X_K}$, we have $q^*q_*\mathcal{E}\in\mathcal{V}_{X_K}$.
\end{Lemma}

\begin{proof}
For any $\mathcal E\in \mathcal V_{X_K}$, by Lemma~\ref{inseparable-filtration}, there exists a filtration of $q^*q_*\mathcal E$:
$$q^*q_*\mathcal E=J^0(\mathcal E) \hookleftarrow J^1(\mathcal E)\hookleftarrow\cdots \hookleftarrow J^N(\mathcal E)=0$$
such that each factor $J^j(\mathcal E)/J^{j+1}(\mathcal E) \cong \mathcal E^{\oplus m_j}\in\mathcal{V}_{X_K}$ for $0\leq j< N$. Consider the following exact sequence
$$0\rightarrow J^{N-1}(\mathcal{E})\cong \mathcal{E}^{\oplus m_{N-1}}\rightarrow J^{N-2}(\mathcal{E})\rightarrow J^{N-2}(\mathcal{E})/J^{N-1}(\mathcal{E})\cong \mathcal{E}^{\oplus m_{N-2}}\rightarrow 0$$
Since $\mathcal V_{X_K}$ is closed under taking extensions, we obtain $J^{N-2}(\mathcal{E})\in\mathcal{V}_{X_K}$. Repeating this method, we have $q^*q_*\mathcal{E}\in\mathcal V_{X_K}$.
\end{proof}

\begin{Proposition}\label{fieldfiniteequiv}
    Let $k$ be a field, $K/k$ a finite field extension, $X$ a connected quasi-compact and quasi-separated scheme over $k$, $x_K\in X_K(K)$ lying over $x\in X(k)$, $\mathcal{C}_X$ and $\mathcal{C}_{X_K}$ the Tannakian categories on $X$ and $X_K$ respectively. Suppose for any $E\in\mathcal{C}_X$, $E\otimes_k K\in\mathcal{C}_{X_K}$. Then
    \begin{enumerate}
        \item Let $K/k$ be finite Galois. The following conditions are equivalent:
        \begin{enumerate}
            \item The natural homomorphism $\pi(\mathcal C_{X_K},x_K)\rightarrow\pi(\mathcal C_X,x)_K$ is an isomorphism.
            \item \begin{enumerate}
                \item For any $E\in\Vect(X)$, $E\in\mathcal C_X$ iff $E\otimes_k K\in\mathcal C_{X_K}$.
            \item For any $\mathcal{E}\in\mathcal{C}_{X_K}$ and any automorphism $\tilde{g}:X_K\rightarrow X_K$ induced by $g\in\Gal(K/k)$, $\tilde{g}^*\mathcal{E}\in\mathcal{C}_{X_K}$.
            \end{enumerate}
        \end{enumerate}
        \item Let $K/k$ be finite purely inseparable and $\mathcal{C}_{X_K}$ saturated. The following conditions are equivalent:
        \begin{enumerate}
            \item The natural homomorphism $\pi(\mathcal C_{X_K},x_K)\rightarrow\pi(\mathcal C_X,x)_K$ is an isomorphism.
            \item For any $E\in\Vect(X)$, $E\in\mathcal C_X$ iff $E\otimes_k K\in\mathcal C_{X_K}$.
        \end{enumerate}
    \end{enumerate}
\end{Proposition}

\begin{proof}
    (1) It follows by Theorem~\ref{fieldfinite} and Lemma~\ref{fieldfinitegalois}.

    (2) It follows by Theorem~\ref{fieldfinite} and Lemma~\ref{fieldfiniteinseparable}.
\end{proof}

\subsection{Base change of Tannaka group schemes over algebraically closed field}

\begin{Proposition}\label{fieldalgclosed}
    Let $k$ be an algebraically closed field, $K/k$ an extension of fields, $X$ a connected quasi-compact and quasi-separated scheme over $k$, $x_K\in X_K(K)$ lying over $x\in X(k)$, $\mathcal{C}_X$ and $\mathcal{C}_{X_K}$ the Tannakian categories on $X$ and $X_K$ respectively. Suppose for any $E\in\mathcal{C}_X$, $E\otimes_k K\in\mathcal{C}_{X_K}$. Then
    \begin{enumerate}
        \item The following conditions are equivalent:
        \begin{enumerate}
            \item For any simple $E\in\mathcal{C}_X$, $E\otimes_k K$ is simple in $\mathcal{C}_{X_K}$.
            \item The natural homomorphism $\pi(\mathcal{C}_{X_K},x_K)\rightarrow \pi(\mathcal{C}_X,x)_K$ is faithfully flat.
        \end{enumerate}
        \item If the natural homomorphism $\pi(\mathcal{C}_{X_K},x_K)\rightarrow \pi(\mathcal{C}_X,x)_K$ is an isomorphism, then for any simple object $\mathcal{E}\in\mathcal{C}_{X_K}$, there exists $E\in\mathcal{C}_X$ such that $\mathcal{E}\cong E\otimes_k K$.
    \end{enumerate}
\end{Proposition}

\begin{proof}
    (1) $(a)\Rightarrow (b)$ Let $E\in \mathcal{C}_{X}$, $\mathcal{E}\in\mathcal{C}_{X_K}$ of rank $r$ and $\mathcal{E}\hookrightarrow E\otimes_k K\in\mathcal{C}_{X_K}$. Note that $$\wedge ^{r-1}\mathcal{E}\hookrightarrow \wedge ^{r-1}E\otimes_k K, \det \mathcal{E}\cong \wedge ^{r}\mathcal{E}\hookrightarrow \wedge ^{r}E\otimes_k K.$$ Since $\det \mathcal{E}$ is of rank $1$, so simple in $\mathcal{C}_{X_K}$. Consider the Jordan H\"older filtration of $\wedge^{r}E$ in $\mathcal{C}_X$:
    $$0\hookrightarrow E_1\hookrightarrow \cdots\hookrightarrow E_n=\wedge^{r}E,$$
    such that $E_{i+1}/E_i$ is simple in $\mathcal{C}_X$ for any $i$. Tensoring with $K$, we obtain a filtration of $\wedge^{r}E\otimes_k K$:
    $$0\hookrightarrow E_1\otimes_k K\hookrightarrow \cdots\hookrightarrow E_n\otimes_k K=\wedge^{r}E\otimes_k K,$$
    and $(E_{i+1}\otimes_k K)/(E_i\otimes_k K)\cong (E_{i+1}/E_i)\otimes_k K$ for any $i$. By assumption, $(E_{i+1}/E_i)\otimes_k K$ is simple in $\mathcal C_{X_K}$. Since $\det \mathcal{E}\hookrightarrow \wedge^{r}E\otimes_k K$ is simple, there exists $(E_{i_0+1}/E_{i_0})\otimes_k K\cong \det \mathcal{E}$. Then 
    $$\mathcal{E}^\vee\cong \wedge ^{r-1}\mathcal{E}\otimes_{\mathcal{O}_{X_K}} (\det \mathcal{E})^\vee\hookrightarrow ((\wedge ^{r-1}E)\otimes_{\mathcal{O}_X}(E_{i_0+1}/E_{i_0})^{\vee})\otimes_k K.$$
    Hence the natural homomorphism $\pi(\mathcal{C}_{X_K},x_K)\rightarrow \pi(\mathcal{C}_X,x)_K$ is faithfully flat by [Proposition~\ref{fieldgeneral}, (3)].

    $(b)\Rightarrow (a)$ For any simple $E\in\mathcal{C}_X$, since $k$ is algebraically closed, we have $\End_{\mathcal O_{X_K}}(E\otimes_k K)\cong K$. Let $\mathcal F\hookrightarrow E\otimes_k K$ be a nonzero subobject in $\mathcal C_{X_K}$. By [Proposition~\ref{fieldgeneral}, (3)], there exists $F\in\mathcal C_X$ such that $\mathcal F^\vee\hookrightarrow F^\vee\otimes_k K$. Then we obtain a nonzero morphism $f:F\otimes_k K\twoheadrightarrow\mathcal F \hookrightarrow E\otimes_k K$. Note that $\Hom_{\mathcal O_{X_K}} (F\otimes_k K,E\otimes_k K) \cong \Hom_{\mathcal O_X}(F,E)\otimes_k K$, so we have $f=\sum_{i=1}^m f_i\otimes a_i$, where $f_i\in\Hom_{\mathcal O_X}(F,E)$ and  $a_i\in K$. Consider the morphism $(f_1,\ldots,f_m):F\rightarrow E^{\oplus m}$. Since $E$ is simple, $E^{\oplus m}$ is semisimple and the image of this morphism is isomorphic to $E^{\oplus r}$ for some $r\geq1$. It follows that $f$ factors through $F\otimes_k K \twoheadrightarrow (E\otimes_k K)^{\oplus r} \rightarrow E\otimes_k K$. Since $\End_{\mathcal O_{X_K}}(E\otimes_k K)\cong K$, the last morphism is given by a row of scalars in $K$. This row is nonzero because $f\ne0$, so the last morphism is a split surjection. Thus $f$ is surjective. Since $\Ima(f)=\mathcal F$, we obtain $\mathcal F=E\otimes_k K$. Therefore, $E\otimes_k K$ is simple in $\mathcal C_{X_K}$.
    
    (2) Suppose the natural homomorphism $\pi(\mathcal{C}_{X_K},x_K)\rightarrow \pi(\mathcal{C}_X,x)_K$ is an isomorphism. For any simple $F\in\mathcal C_X$, $F\otimes_k K$ is simple in $\mathcal{C}_{X_K}$ by (1). Let $\mathcal{E}\in\mathcal{C}_{X_K}$ be a simple object, by Theorem~\ref{fieldgeneralcriteria} there exists $E\in\mathcal{C}_X$ such that $\mathcal{E}\hookrightarrow E\otimes_k K$. Consider the Jordan H\"older filtration of $E$ in $\mathcal{C}_X$:
    $$0\hookrightarrow E_1\hookrightarrow \cdots \hookrightarrow E_n=E,$$
    and denote $E^i:=E_{i+1}/E_i$ for $0\leq i\leq n-1$. Tensoring with $K$, we obtain a filtration of $E\otimes_k K$:
    $$0\hookrightarrow E_1\otimes_k K\hookrightarrow \cdots \hookrightarrow E_n\otimes_k K=E\otimes_k K.$$
    Then $E^i\otimes_k K$ is simple in $\mathcal{C}_{X_K}$. Since $\mathcal{E}\hookrightarrow E\otimes_k K$ is simple in $\mathcal{C}_{X_K}$, it follows that $\mathcal{E}$ is a factor in the Jordan H\"older filtration of $E\otimes_k K$ in $\mathcal{C}_{X_K}$, i.e. there exists $E^i$ such that  $\mathcal{E}\cong E^i\otimes_k K$.
\end{proof}

\begin{Theorem}\label{fieldalgclosedthm}
    Let $k$ be an algebraically closed field, $K/k$ an extension of fields, $X$ a connected quasi-compact and quasi-separated scheme over $k$, $x_K\in X_K(K)$ lying over $x\in X(k)$, $\mathcal{C}_X$ and $\mathcal{C}_{X_K}$ the saturated Tannakian categories on $X$ and $X_K$ respectively. Suppose for any $E\in\mathcal{C}_X$, $E\otimes_k K\in\mathcal{C}_{X_K}$. Then the following conditions are equivalent:
        \begin{enumerate}
            \item The natural homomorphism $\pi(\mathcal{C}_{X_K},x_K)\rightarrow \pi(\mathcal{C}_X,x)_K$ is an isomorphism.
            \item For any simple object $\mathcal{E}\in\mathcal{C}_{X_K}$, there exists $E\in\mathcal{C}_X$ such that $\mathcal{E}\cong E\otimes_k K$.
        \end{enumerate} 
\end{Theorem}

\begin{proof}
    $(1)\Rightarrow (2)$ It follows by [Proposition~\ref{fieldalgclosed}, (2)].
    
    $(2)\Rightarrow (1)$ Let $F\in\mathcal C_X$ and $\mathcal F\hookrightarrow F\otimes_k K$ be a nonzero subobject in $\mathcal C_{X_K}$, with $r=\rk\mathcal F$. Since $\det\mathcal F$ is a simple object of $\mathcal C_{X_K}$, (2) implies that there exists $L\in\mathcal C_X$ such that $L\otimes_k K\cong\det\mathcal F$. Hence
\[
    \mathcal F^\vee \cong \wedge^{r-1}\mathcal F\otimes_{\mathcal O_{X_K}} (\det\mathcal F)^\vee \hookrightarrow \left(\wedge^{r-1}F\otimes_{\mathcal O_X}L^\vee\right) \otimes_k K.
\]
Since $\wedge^{r-1}F\otimes_{\mathcal O_X}L^\vee \in\mathcal C_X$, the natural homomorphism is faithfully flat by [Proposition~\ref{fieldgeneral}, (3)].
    
    For any $\mathcal{E}\in \mathcal{C}_{X_K}$, consider the Jordan H\"older filtration of $\mathcal{E}$ in $\mathcal{C}_{X_K}$:
    $$0\hookrightarrow \mathcal{E}_1\hookrightarrow \cdots\hookrightarrow \mathcal{E}_n=\mathcal{E},$$
    where each $\mathcal{E}^i=\mathcal{E}_{i+1}/\mathcal{E}_i$ is simple in $\mathcal{C}_{X_K}$. Then we have the following exact sequence
    $$0\rightarrow \mathcal{E}_1\rightarrow \mathcal{E}_2\rightarrow \mathcal{E}_2/\mathcal{E}_1\rightarrow 0.$$
    By assumption, there exists $E_1,E''_2\in\mathcal{C}_X$ such that $\mathcal{E}_1\cong E_1\otimes_k K$ and $\mathcal{E}_2/\mathcal{E}_1\cong E''_2\otimes_k K$. Applying Lemma~\ref{subquotient}, $\mathcal{E}_2$ is a subquotient of $E'_2\otimes_k K$ for some $E'_2\in\mathcal{C}_X$. Then by  [Proposition~\ref{fieldgeneral}, (3)], we have $\mathcal{E}_2$ is a subobject of $E_2\otimes_k K$ for some $E_2\in\mathcal{C}_X$. Then we consider the following exact sequence
    $$0\rightarrow \mathcal{E}_{2}\rightarrow \mathcal{E}_3\rightarrow \mathcal{E}_3/\mathcal{E}_{2}\rightarrow 0.$$
    By assumption, there exist $E''_3\in\mathcal{C}_{X}$ such that $\mathcal{E}_3/\mathcal{E}_{2}\cong E''_3\otimes_k K$. Consider the following pushout diagram
    \[
        \begin{tikzcd}
            0\arrow[r]&\mathcal{E}_{2}\arrow[r]\arrow[d,hook]&\mathcal{E}_3\arrow[r]\arrow[d,hook]&\mathcal{E}_3/\mathcal{E}_{2}\arrow[r]\arrow[d,"\cong"]&0\\
            0\arrow[r]&E_2\otimes_k K\arrow[r]&\mathcal{E}'_3\arrow[ul, phantom, "\lrcorner", very near start]\arrow[r]&E''_3\otimes_k K\arrow[r]&0
        \end{tikzcd}
    \]
    Applying Lemma~\ref{subquotient}, $\mathcal{E}'_3$ is a subquotient of $E'_3\otimes_k K$ for some $E'_3\in\mathcal{C}_X$. Then by [Proposition~\ref{fieldgeneral}, (3)], $\mathcal{E}_3$ is a subobject of $E_3\otimes_k K$ for some $E_3\in\mathcal{C}_X$. Repeating this method, for any $i$, there exists $E_i\in\mathcal{C}_X$ such that $\mathcal{E}_i\hookrightarrow E_i\otimes_k K$. In particular, $\mathcal{E}\hookrightarrow E\otimes_k K$ for some $E\in\mathcal{C}_X$. According to Theorem~\ref{fieldgeneralcriteria}, the natural homomorphism $\pi(\mathcal{C}_{X_K},x_K)\rightarrow \pi(\mathcal{C}_X,x)_K$ is an isomorphism.
\end{proof}

\begin{Proposition}\label{algclosednotiso}
    Let $k$ be an algebraically closed field, $K/k$ an extension of fields, $X$ a connected quasi-compact and quasi-separated scheme over $k$, $x_K\in X_K(K)$ lying over $x\in X(k)$, $\mathcal{C}_X\subseteq\mathcal C'_X$ and $\mathcal{C}_{X_K}\subseteq\mathcal C'_{X_K}$ the Tannakian categories on $X$ and $X_K$ respectively. Suppose the following conditions:
    \begin{enumerate}
        \item For any $E'\in\mathcal{C}'_X$, $E'\otimes_k K\in\mathcal{C}'_{X_K}$.
        \item For any $E\in\Vect(X)$, $E\in\mathcal{C}_X$ iff $E\otimes_k K\in\mathcal{C}_{X_K}$.
        \item There exists a simple object $\mathcal{E}_0\in\mathcal{C}_{X_K}$ such that $\mathcal{E}_0\not\cong E\otimes_k K$ for any $E\in\mathcal{C}_X$.
    \end{enumerate}
    Then the natural homomorphism $\pi(\mathcal{C}'_{X_K},x_K)\rightarrow\pi(\mathcal{C}'_X,x)_K$ is not an isomorphism.
\end{Proposition}

\begin{proof}
    Suppose the natural homomorphism $\pi(\mathcal{C}'_{X_K},x_K)\rightarrow\pi(\mathcal{C}'_X,x)_K$ is an isomorphism, then for any simple object $\mathcal{E}'\in\mathcal{C}'_{X_K}$, there exists $E'\in\mathcal{C}'_{X}$ such that $\mathcal{E}'\cong E'\otimes_k K$ by [Proposition~\ref{fieldalgclosed}, (2)]. There exists simple object $\mathcal{E}_0\in \mathcal{C}_{X_K}\subseteq \mathcal{C}'_{X_K}$ such that $\mathcal{E}_0\not\cong E\otimes_k K$ for any $E\in\mathcal{C}_{X}$. Note that $\mathcal{E}_0$ being simple in $\mathcal{C}_{X_K}$ is equivalent to being simple in $\mathcal{C}'_{X_K}$. If $\mathcal{E}_0\cong E_0\otimes_k K$ for some $E_0\in\mathcal{C}'_X$, then $E_0\in\mathcal{C}_X$ by condition (2), which yields a contradiction.
\end{proof}

\section{Base change of fundamental group schemes}

Statements involving \(\mathcal C^{Loc}\) or its saturations are understood to be in positive characteristic.

\begin{Definition}
Let $k$ be a field, $X$ a geometrically reduced connected scheme proper over $k$, $x\in X(k)$, $E$ a vector bundle on $X$ of rank $r$. If $k$ is of positive characteristic, then $F_X:X\rightarrow X$ is the absolute Frobenius morphism. Then $E$ is said to be
\begin{itemize}
    \item \textit{numerically flat}, if both $E$ and $E^\vee$ are nef.
    \item \textit{Nori semistable}, if for any algebraically closed field $K/k$, any smooth projective curve $C$ over $K$, and any $K$-morphism $f:C\rightarrow X_K$, the vector bundle $f^*(E\otimes_k K)$ is semistable of degree $0$.
    \item \textit{finite}, if there exist $f(t)\neq g(t)\in\mathbb{N}[t]$ such that $f(E)\cong g(E)$, where
    $$h(E):=\bigoplus\limits_{i=0}^m(E^{\otimes i})^{\oplus n_i}\text{ for any }h(t)=\sum\limits_{i=0}^mn_it^i\in\mathbb{N}[t].$$
    Denote the set of all finite vector bundles on $X$ by $F(X)$.
    \item \textit{essentially finite}, if there exists $E_1\hookrightarrow E_2\hookrightarrow F\in\Vect(X)$ such that $E\cong E_2/E_1$, where $E_1,E_2$ are numerically flat and $F$ is finite.
    \item \textit{Frobenius trivial}, if there exists a positive integer $n$ such that $F_X^{n*} E\cong \mathcal{O}_X^{\oplus r}$.
    \item \textit{\'etale trivializable}, if there exists a finite \'etale covering $\phi: P\rightarrow X$ such that $\phi^* E\cong \mathcal{O}_P^{\oplus r}$.
    \item \textit{unipotent}, if there exists a filtration
    $0=E_0\hookrightarrow E_1\hookrightarrow \cdots\hookrightarrow E_n=E$
    such that $E_{i+1}/E_i\cong\mathcal{O}_X$ for any $i$.
\end{itemize}
We have the following Tannakian categories:
    \begin{itemize}
        \item $\mathcal{C}^{NF}(X)$ (or sometimes $\mathcal{C}^{S}(X)$): objects consist of numerically flat bundles on $X$.
        \item $\mathcal{C}^{N}(X)$: objects consist of essentially finite bundles on $X$.
        \item $\mathcal{C}^{Loc}(X)$: objects consist of Frobenius trivial bundles on $X$.
        \item $\mathcal{C}^{\acute{e}t}(X)$: objects consist of \'etale trivializable bundles on $X$.
        \item $\mathcal{C}^{uni}(X)$: objects consist of unipotent bundles on $X$.
    \end{itemize}
We have the following Tannaka group schemes:
\begin{itemize}
        \item $\pi^{S}(X,x):=\pi(\mathcal{C}^{NF}(X),x)$, called the \textit{S-fundamental group scheme}.
        \item $\pi^{N}(X,x):=\pi(\mathcal{C}^{N}(X),x)$, called the \textit{Nori fundamental group scheme}.
        \item $\pi^{EN}(X,x):=\pi(\overline{\mathcal{C}^{N}(X)},x)$, called the \textit{extended Nori fundamental group scheme}.
        \item $\pi^{Loc}(X,x):=\pi(\mathcal{C}^{Loc}(X),x)$, called the \textit{local fundamental group scheme}.
        \item $\pi^{ELoc}(X,x):=\pi(\overline{\mathcal{C}^{Loc}(X)},x)$, called the \textit{extended local fundamental group scheme}.
        \item $\pi^{\acute{e}t}(X,x):=\pi(\mathcal{C}^{\acute{e}t}(X),x)$, called the \textit{\'etale fundamental group scheme}.
        \item $\pi^{E\acute{e}t}(X,x):=\pi(\overline{\mathcal{C}^{\acute{e}t}(X)},x)$, called the \textit{extended \'etale fundamental group scheme}.
        \item $\pi^{uni}(X,x):=\pi(\mathcal{C}^{uni}(X),x)$, called the \textit{unipotent fundamental group scheme}.
    \end{itemize}
\end{Definition}

\begin{Remark}
    Let $k$ be a field, $X$ a geometrically reduced connected scheme proper over $k$. We adopt the definition of Nori semistability given in \cite[Definition~4.1]{ABETZ19}, which agrees with the usual definition when $k$ is algebraically closed. With this definition, a vector bundle $E$ on $X$ is numerically flat if and only if it is Nori semistable; see \cite[Theorem~2.13]{FuLa22}.
\end{Remark}

\begin{Remark}
    Let $\characteristic k=0$. Otabe \cite{Ota17} showed that the natural homomorphism $\pi^{EN}(X,x)\rightarrow\pi^N(X,x)$ need not be an isomorphism. For example, if $X$ is an elliptic curve over an algebraically closed field of characteristic zero, then $\pi^{EN}(X,x) \cong\pi^N(X,x)\times\pi^{uni}(X,x)$, where $\pi^{uni}(X,x)$ is nontrivial \cite[Theorem~3.10]{Ota17}. 

    Let $\characteristic k>0$. Amrutiya proved that $\pi^{EN}(X,x)\cong\pi^N(X,x)$ when $X$ is an elliptic curve over a perfect field \cite[Theorem~4.3]{Amr20}. The corresponding result for abelian varieties over algebraically closed fields was also established by Adroja and Amrutiya \cite[Corollary~3.5]{AdAm25}. More generally, the result of Tonini and Zhang \cite[Corollary~II]{ToZh19} implies that $\overline{\mathcal{C}^{N}(X)}=\mathcal C^N(X)$ for any field $k$ of $\characteristic k>0$ and any geometrically reduced connected scheme proper over $k$ with a rational point; see also \cite[Theorems~10.5 and~10.9]{BoVi19}.
\end{Remark}

\begin{Remark}\label{etnosaturated}
    Let $k$ be a field, $X$ a geometrically reduced connected scheme proper over $k$. By definition, $\mathcal{C}^{NF}(X)$ and $\mathcal{C}^{uni}(X)$ are saturated. Moreover, $\mathcal{C}^{N}(X)$ is saturated when $\characteristic k>0$. However, the category $\mathcal{C}^{Loc}(X)$ and $\mathcal{C}^{\acute{e}t}(X)$ are not saturated in general. 

    Let $k$ be an algebraically closed field of characteristic $p>0$, $C$ an ordinary elliptic curve over $k$, and $o\in C(k)$ its origin. Since $\operatorname{Ext}^1_{\mathcal O_C} (\mathcal O_C,\mathcal O_C) \cong H^1(C,\mathcal O_C)\cong k$, we choose a nonzero class $c\in H^1(C,\mathcal O_C)$ and consider the corresponding non-split extension
    \[
        0\rightarrow\mathcal O_C\rightarrow\mathcal E
        \rightarrow\mathcal O_C\rightarrow0.
    \]
    Since $\mathcal O_C\in\mathcal C^{\mathrm{Loc}}(C)$, we have $\mathcal E\in\mathcal C^{ELoc}(C)$. Let $F_C:C\rightarrow C$ be the absolute Frobenius morphism. Since $C$ is ordinary, the induced semilinear map $F_C^*:H^1(C,\mathcal O_C) \rightarrow H^1(C,\mathcal O_C)$ is bijective. Hence $(F_C^*)^n(c)\neq0$ for every $n\geq1$. Pulling back the above extension by $F_C^n$, we obtain
    \[
        0\rightarrow\mathcal O_C
        \rightarrow F_C^{n*}\mathcal E
        \rightarrow\mathcal O_C\rightarrow0,
    \]
    whose extension class is $(F_C^*)^n(c)$. Thus this extension is non-split for every $n\geq1$. The connecting homomorphism $H^0(C,\mathcal O_C)\rightarrow H^1(C,\mathcal O_C)$ sends $1$ to $(F_C^*)^n(c)$ and is therefore injective. The associated long exact sequence gives $h^0(C,F_C^{n*}\mathcal E)=1$. Since $h^0(C,\mathcal O_C^{\oplus2})=2$, we have $F_C^{n*}\mathcal E\not\cong\mathcal O_C^{\oplus2}$ for every $n\geq1$. Consequently, $\mathcal{C}^{Loc}(C)\subsetneq \mathcal{C}^{ELoc}(C)$.

    Let $k$ be an algebraically closed field of characteristic
    $p>0$, $C$ a supersingular elliptic curve over $k$, and
    $o\in C(k)$ its origin. Since $\operatorname{Ext}^1_{\mathcal O_C} (\mathcal O_C,\mathcal O_C) \cong H^1(C,\mathcal O_C)\cong k$, we choose a nonzero class $c\in H^1(C,\mathcal O_C)$ and consider the corresponding non-split extension
    \[
        0\rightarrow\mathcal O_C\rightarrow\mathcal E
        \rightarrow\mathcal O_C\rightarrow0.
    \]
    Since $\mathcal O_C\in\mathcal C^{\acute{e}t}(C)$, we have $\mathcal E\in\mathcal C^{E\acute{e}t}(C)$. The connecting homomorphism $H^0(C,\mathcal O_C)\rightarrow H^1(C,\mathcal O_C)$ sends $1$ to $c\neq0$ and is therefore injective. Hence the associated long exact sequence gives $h^0(C,\mathcal E)=1$, so that $\mathcal E\not\cong\mathcal O_C^{\oplus2}$. Let $F_C:C\rightarrow C$ be the absolute Frobenius morphism. Since $C$ is supersingular and $\dim_k H^1(C,\mathcal O_C)=1$, the induced semilinear map $F_C^*:H^1(C,\mathcal O_C)\rightarrow H^1(C,\mathcal O_C)$ is zero. Thus the pullback extension $0\rightarrow\mathcal O_C\rightarrow F_C^*\mathcal E \rightarrow\mathcal O_C\rightarrow0$ has extension class $F_C^*(c)=0$ and splits. Consequently, $F_C^*\mathcal E\cong\mathcal O_C^{\oplus2}$, so that $\mathcal E\in\mathcal C^{\mathrm{Loc}}(C)$. Suppose that $\mathcal E\in\mathcal C^{\acute{e}t}(C)$, then we obtain $\mathcal E\cong\mathcal O_C^{\oplus2}$, a contradiction. Therefore, $\mathcal{C}^{\acute{e}t}(C)\subsetneq \mathcal{C}^{E\acute{e}t}(C)$.
\end{Remark}

\begin{Lemma}[{\cite[Lemma~1.4.3]{CLM22}}]\label{pullbacknef}
    Let $f:X\rightarrow Y$ be a morphism of proper algebraic spaces over $k$, $E\in\Vect(Y)$. Then
    \begin{enumerate}
        \item If $E$ is nef, then $f^*E$ is nef.
        \item If $f$ is surjective, then $E$ is nef iff $f^*E$ is nef.
    \end{enumerate} 
\end{Lemma}

\begin{Lemma}[{\cite[Lemma~1.4.4]{CLM22}}]\label{basefieldnef}
    Let $X$ be a proper algebraic space over $k$, $E\in\Vect(X)$. Then $E$ is nef iff for every field extension $k\subseteq k'$, the pullback $E\otimes_k k'$ on $X_{k'}$ is nef.
\end{Lemma}

\begin{Proposition}\label{generalvbproperty}
    Let $k$ be a field, $K/k$ a field extension, $X$ a geometrically reduced connected scheme proper over $k$, $x\in X(k)$, $E\in\Vect(X)$. 
    Then 
    \begin{enumerate}
        \item For $*\in\{S,Loc,uni\}$, $E\in\mathcal C^*(X)$ iff $E\otimes_k K\in\mathcal C^*(X_K)$.
        \item $E\in F(X)$ iff $E\otimes_k K\in F(X_K)$.
        \item For $*\in \{N,\acute{et},EN,ELoc,E\acute{et}\}$, if $E\in \mathcal C^*(X)$, then $E\otimes_k K\in\mathcal C^*(X_K)$.
        \item Let $K/k$ be algebraic. Then for $*\in\{N,\acute{et}\}$, $E\in\mathcal C^*(X)$ iff $E\otimes_k K\in\mathcal C^*(X_K)$.
    \end{enumerate}
\end{Proposition}

\begin{proof}
    (1) For $*=S$: If $E\in\mathcal{C}^{NF}(X)$, then it follows immediately by Lemma~\ref{basefieldnef}. Conversely, if $E\otimes_k K\in\mathcal{C}^{NF}(X_K)$. Suppose that $\mathcal O_{\mathbf P(E)}(1)_K$ is nef. For any integral closed curve $C\subseteq {\mathbf P(E)}$, cohomology and flat base change give
\[
    \deg_C(\mathcal O_{\mathbf P(E)}(1)|_C)
    =\chi(C,\mathcal O_{\mathbf P(E)}(1)|_C)-\chi(C,\mathcal O_C)
    =\deg_{C_K}(\mathcal O_{\mathbf P(E)}(1)_K|_{C_K}).
\]
The degree on $C_K$ is the sum of the degrees on its one-dimensional irreducible components, endowed with their reduced structures, weighted by the positive lengths at their generic points. Since $\mathcal O_{\mathbf P(E)}(1)_K$ is nef, the degree on each component is nonnegative. Hence $\deg_C(\mathcal O_{\mathbf P(E)}(1)|_C)\geq0$ for every integral closed curve $C\subseteq {\mathbf P(E)}$. Therefore $\mathcal O_{\mathbf P(E)}(1)$ is nef, and so $E$ is nef. Applying the same argument to $E^\vee$, and using $(E^\vee)\otimes_k K\cong(E\otimes_k K)^\vee$, we conclude that $E\in\mathcal{C}^{NF}(X)$.

    For $*=Loc$: If $F_X^{n*}E\cong \mathcal{O}_X^{\oplus r}$, then $F_{X_K}^{n*}(E\otimes_k K)\cong (F_X^{n*}E)\otimes_k K\cong \mathcal{O}_{X_K}^{\oplus r}$, so that $E\otimes_k K\in\mathcal{C}^{Loc}(X_K)$. Conversely, if $E\otimes_k K\in\mathcal{C}^{Loc}(X_K)$, there exists $n\in\mathbb{N}_+$, s.t. $(F_X^{n*}E)\otimes_k K\cong F_{X_K}^{n*}(E\otimes_k K)\cong \mathcal{O}_X^{\oplus r}\otimes_k K$. Then $F_X^{n*}E\cong \mathcal{O}_X^{\oplus r}$ by Lemma~\ref{NoetherDeuring}, so that $E\in\mathcal{C}^{Loc}(X)$. 

    For $*=uni$: If $E\in\mathcal{C}^{uni}(X)$, then $E\otimes_k K\in\mathcal{C}^{uni}(X_K)$ is obvious.
    Conversely, we will use induction on the rank of $E$. If $E\otimes_k K\in\mathcal{C}^{uni}(X_K)$, then $E\in\mathcal{C}^{NF}(X)$. Since $\mathcal{O}_{X_K}\hookrightarrow E\otimes_k K$, we have $H^0(X_K,E\otimes_k K)\neq 0$. Then by $H^0(X_K,E\otimes_k K)\cong H^0(X,E)\otimes_k K$, we obtain $H^0(X,E)\neq 0$ and an exact sequence in $\mathcal{C}^{NF}(X)$:
    $$0\rightarrow \mathcal{O}_X\rightarrow E\rightarrow E/\mathcal{O}_X\rightarrow 0.$$
    If $\rk E=1$, then $E\otimes_k K\cong \mathcal{O}_X\otimes_k K$ and by Lemma~\ref{NoetherDeuring}, we have $E\cong \mathcal{O}_X$.
    If $\rk E=2$, consider the exact sequence in $\mathcal{C}^{NF}(X)$:
    $$0\rightarrow \mathcal{O}_X\rightarrow E\rightarrow E/\mathcal{O}_X\rightarrow 0.$$
    Tensoring with $K$, we obtain an exact sequence in $\mathcal{C}^{uni}(X_K)$:
    $$0\rightarrow \mathcal{O}_{X_K}\rightarrow E\otimes_k K\rightarrow (E/\mathcal{O}_X)\otimes_k K\rightarrow 0.$$
    So $(E/\mathcal{O}_X)\otimes_k K\cong \mathcal{O}_{X}\otimes_k K$. It follows that $E/\mathcal{O}_X\cong \mathcal{O}_X$ by Lemma~\ref{NoetherDeuring}, so that $E\in\mathcal{C}^{uni}(X)$.
    Let $\rk E=n>2$ and suppose for any vector bundle $F$ of rank less than $n$, $F\otimes_k K\in\mathcal{C}^{uni}(X_K)$ implies $F\in\mathcal{C}^{uni}(X)$. Consider the Jordan H\"older filtration of $E$ in $\mathcal{C}^{NF}(X)$:
    $$0\hookrightarrow E_1\hookrightarrow \cdots\hookrightarrow E_m=E,$$
    where $E_{i+1}/E_i$ is simple in $\mathcal{C}^{NF}(X)$ for any $i$. Consider the exact sequence
    $$0\rightarrow E_{m-1}\rightarrow E\rightarrow E/E_{m-1}\rightarrow 0.$$
    Tensoring with $K$, we obtain an exact sequence
    $$0\rightarrow E_{m-1}\otimes_k K\rightarrow E\otimes_k K\rightarrow E/E_{m-1}\otimes_k K\rightarrow 0,$$
    where $E_{m-1}\otimes_k K\hookrightarrow E\otimes_k K$ in $\mathcal{C}^{NF}(X_K)$. So $E_{m-1}\otimes_k K\in\mathcal{C}^{uni}(X_K)$ and $E/E_{m-1}\otimes_k K\in\mathcal{C}^{uni}(X_K)$. Then by hypothesis induction, $E_{m-1},E/E_{m-1}\in\mathcal{C}^{uni}(X)$. So $E\in\mathcal{C}^{uni}(X)$ since $\mathcal{C}^{uni}(X)$ is saturated.

    (2) For finite: If $E\in F(X)$, then there exists distinct $f,g\in \mathbb{N}[t]$ such that $f(E)\cong g(E)$. Then $$f(E\otimes_k K)\cong f(E)\otimes_k K\cong g(E)\otimes_k K\cong g(E\otimes_k K),$$
    i.e. $E\otimes_k K\in F(X_K)$. Conversely, if $E\otimes_k K\in F(X_K)$, then there exist distinct $f,g\in \mathbb{N}[t]$ such that
    $$f(E)\otimes_k K\cong f(E\otimes_k K)\cong g(E\otimes_k K)\cong g(E)\otimes_k K.$$
    Then we have $f(E),g(E)\in\mathcal{C}^{NF}(X)$ by (1). Applying Lemma~\ref{NoetherDeuring}, we have $f(E)\cong g(E)$, so $E\in F(X)$.

    (3) For $*=N$: For any $E\in\mathcal{C}^N(X)$, we have $E\cong F_1/F_2$, where $F_2\hookrightarrow F_1\hookrightarrow F$, $F_2,F_1\in\mathcal{C}^{NF}(X)$ and $F$ is a finite bundle on $X$. Then $F_1\otimes_k K,F_2\otimes_k K\in\mathcal{C}^{NF}(X_K)$, $F\otimes_k K$ is finite by (1) and (2). It follows that $E\otimes_k K\cong (F_1/F_2)\otimes_k K\cong (F_1\otimes_k K)/(F_2\otimes_k K)\in\mathcal{C}^N(X_K)$.
    
% Target: generaldescent, (3), *=etale. Source: generalet (1).

    For $*=\acute{e}t$: Since $E\in\mathcal{C}^{\acute{e}t}(X)$, there exists a finite \'etale covering $\pi:P\rightarrow X$ such that $\pi^*E\cong \mathcal{O}_P^{\oplus n}$. It follows that $\pi_K:P_K\rightarrow X_K$ is also a finite \'etale covering and 
    $(\pi_K)^*(E\otimes_k K)\cong \pi^* E\otimes_k K\cong \mathcal{O}_{P_K}^{\oplus n},$
    i.e. $E\otimes_k K\in\mathcal{C}^{\acute{e}t}(X_K)$.

% NO SEPARATE SOURCE PROOF for target (3), *=EN.
% Target: generaldescent, (3), *=ELoc,Eetale. Source: generalsaturaed (1).

    For $*\in\{EN,ELoc,E\acute{e}t\}$: By Proposition~\ref{equivsurjective}.

    (4) For $*=N$: If $E\in\mathcal{C}^{N}(X)$, then $E\otimes_k K\in\mathcal{C}^{N}(X_K)$ by (3).
    
    Conversely, if $E\otimes_k K\in\mathcal C^N(X_K)$. Since $X$ is geometrically reduced, connected and proper over $k$, and $x\in X(k)$, we have $H^0(X,\mathcal O_X)=k$. In particular, $X$ is geometrically connected, and both $X$ and $X_K$ are pseudo-proper and inflexible over their respective fields. If $\operatorname{char}k=0$, then $E\otimes_k K$ is finite. Hence $E$ is finite by (2), so that $E\in\mathcal C^N(X)$. If $\operatorname{char}k>0$. Since $E\otimes_k K$ is essentially finite, there exists a $G$-torsor
    $\pi_K:P_K\rightarrow X_K$ where $G$ is a finite group scheme over
    $K$ such that $\pi_K^*(E\otimes_k K) \cong\mathcal O_{P_K}^{\oplus\rk E}$. In particular, $\pi_K$ is finite and faithfully flat.
    Since $K/k$ is algebraic, there exists a finite extension $L/k$ contained in $K$ and a $G'$-torsor $\pi_L:P_L\rightarrow X_L$ whose base change to $K$ is $\pi_K$, together with an isomorphism
    \[
        \pi_L^*(E\otimes_k L)
        \cong\mathcal O_{P_L}^{\oplus\rk E}.
    \]
    Let $p_L:X_L\rightarrow X$ be the projection. Since $L/k$ is finite, the composition
    \[
        \pi:P_L\xrightarrow{\pi_L}X_L\xrightarrow{p_L}X
    \]
    is finite and faithfully flat, and
    \[
        \pi^*E
        \cong\pi_L^*(E\otimes_k L)
        \cong\mathcal O_{P_L}^{\oplus\rk E}.
    \]
    Moreover, $P_L\times_X P_L$ is finite over $X$, and hence proper over $k$. Consequently, $H^0(P_L\times_X P_L,\mathcal O_{P_L\times_X P_L})$ is a finite-dimensional $k$-algebra. Applying \cite[Corollary~I]{ToZh19}, we obtain $E\in\mathcal C^N(X)$.

    For $*=\acute{e}t$: If $K/k$ is separable and $E\otimes_k K\in\mathcal{C}^{\acute{e}t}(X_K)$, then there exists a finite \'etale covering $\pi_K:P\rightarrow X_K$ such that ${\pi_K}^*(E\otimes_k K)\cong \mathcal{O}_{P}^{\oplus n}$. By \href{https://stacks.math.columbia.edu/tag/0BTV}{Tag 0BTV} and \href{https://stacks.math.columbia.edu/tag/01ZR}{Tag 01ZR}, there exist a finite separable extension $L/k$ contained in $K$ and a finite \'etale covering $\pi_L:Q\rightarrow X_L$ such that $P\cong Q_K$ and $\pi_L^*(E\otimes_k L)\cong\mathcal{O}_Q^{\oplus n}$. Since $L/k$ is finite separable, the projection $p_L:X_L\rightarrow X$ is finite \'etale. Hence the composition $\pi:Q\xrightarrow{\pi_L}X_L\xrightarrow{p_L}X$ is a finite \'etale covering, and $\pi^*E\cong\pi_L^*p_L^*E\cong\pi_L^*(E\otimes_k L)\cong\mathcal{O}_Q^{\oplus n}$. Therefore $E\in\mathcal{C}^{\acute{e}t}(X)$.

If $K/k$ is purely inseparable and $E\otimes_k K\in\mathcal{C}^{\acute{e}t}(X_K)$, then there exists a finite \'etale covering $\pi_K:P\rightarrow X_K$ such that ${\pi_K}^*(E\otimes_k K)\cong \mathcal{O}_{P}^{\oplus n}$. But {\href{https://stacks.math.columbia.edu/tag/0BQN}{Tag 0BQN}} implies that there exists a finite \'etale covering $\pi:Q\rightarrow X$ such that $P\cong Q_K$. Then we have 
$$(\pi^*E)\otimes_k K\cong \pi_K^*(E\otimes_k K)\cong \mathcal{O}_P^{\oplus n}\cong \mathcal{O}_{Q_K}^{\oplus n}.$$
Since $\pi:Q\rightarrow X$ is a finite \'etale covering, we have $\pi_*\mathcal O_Q\in\mathcal C^{\acute{e}t}(X)$ by \cite[Lemma~2.13]{ABETZ19}. Since $E\otimes_k K\in\mathcal C^{\acute{e}t}(X_K)
\subseteq\mathcal C^{NF}(X_K)$, we obtain $E\in\mathcal C^{NF}(X)$ and   $E\otimes_{\mathcal O_X}\pi_*\mathcal O_Q$, $(\pi_*\mathcal O_Q)^{\oplus n} \in\mathcal C^{NF}(X)$. By the projection formula and flat base change, we have
\[
    \begin{aligned}
        (E\otimes_{\mathcal O_X}\pi_*\mathcal O_Q)\otimes_k K
        &\cong (E\otimes_k K)\otimes_{\mathcal O_{X_K}}
        (\pi_K)_*\mathcal O_P\\
        &\cong(\pi_K)_*\pi_K^*(E\otimes_k K)\\
        &\cong(\pi_K)_*(\mathcal O_P^{\oplus n})\\
        &\cong(\pi_*\mathcal O_Q)^{\oplus n}\otimes_k K.
    \end{aligned}
\]
Then by Lemma~\ref{NoetherDeuring}, we obtain $E\otimes_{\mathcal O_X}\pi_*\mathcal O_Q \cong(\pi_*\mathcal O_Q)^{\oplus n}$. Since $\pi$ is faithfully flat, the natural morphism $\mathcal O_X\rightarrow\pi_*\pi^*\mathcal O_X$ is injective. Tensoring with $E$, we obtain $E\hookrightarrow E\otimes_{\mathcal O_X}\pi_*\mathcal O_Q\cong(\pi_*\mathcal O_Q)^{\oplus n}$. Thus $E$ is a subobject of $(\pi_*\mathcal O_Q)^{\oplus n}\in\mathcal C^{\acute{e}t}(X)$ in $\mathcal C^{NF}(X)$. It follows that $E\in\mathcal C^{\acute{e}t}(X)$.

For general algebraic extension, we have sequence of extension of fields  $K/k_{sep}/k$, where $k_{sep}$ is the separable closure of $k$ in $K$. Then $k_{sep}/k$ is a separable extension and $K/k_{sep}$ is a purely inseparable extension. If $E\otimes_k K\in\mathcal{C}^{\acute{e}t}(X_K)$, then $E\otimes_k k_{sep}\in\mathcal{C}^{\acute{e}t}(X_{k_{sep}})$ and $E\in \mathcal{C}^{\acute{e}t}(X)$.
\end{proof}

\begin{Proposition}\label{generalhomomorphism}
    Let $k$ be a field, $K/k$ a field extension, $X$ a geometrically reduced connected scheme proper over $k$, $x_K\in X_K(K)$ lying over $x\in X(k)$, $*\in\{S,N,Loc,\acute{e}t,uni,EN,ELoc,E\acute{e}t\}$. Then there exists a natural homomorphism $\pi^*(X_K,x_K)\rightarrow\pi^*(X,x)_K$.
\end{Proposition}

\begin{proof}
    It follows immediately by Proposition~\ref{generalvbproperty}.
\end{proof}

Nori\cite{Nor82} showed that for arbitrary extension of fields, base change of unipotent fundamental group scheme is an isomorphism. We reprove this isomorphism using the Tannakian duality.

\begin{Proposition}[{\cite[Chapter~IV, Proposition~9]{Nor82}}]\label{generalunipotentbasechange}
    Let $k$ be a field, $K/k$ a field extension, $X$ a geometrically reduced connected scheme proper over $k$, $x_K\in X_K(K)$ lying over $x\in X(k)$. Then the natural homomorphism $\pi^{uni}(X_K,x_K)\rightarrow\pi^{uni}(X,x)_K$ is an isomorphism.
\end{Proposition}

\begin{proof}
    For any $E\in \mathcal{C}^{uni}(X)$ and any subobject $\mathcal{E}'\hookrightarrow E\otimes_k K\in\mathcal{C}^{uni}(X_K)$ of rank $r$. We have
    $$\begin{aligned}
        {\mathcal{E}'}^\vee&\cong \wedge^{r-1}\mathcal{E}'\otimes_{\mathcal{O}_{X_K}} (\det \mathcal{E}')^\vee\\
        &\cong \wedge^{r-1}\mathcal{E}'\otimes_{\mathcal{O}_{X_K}} \mathcal{O}_{X_K}\hookrightarrow \wedge^{r-1}E\otimes_k K.
    \end{aligned}$$
    Then the natural homomorphism $\pi^{uni}(X_K,x_K)\rightarrow \pi^{uni}(X,x)_K$ is faithfully flat by [Proposition~\ref{fieldgeneral}, (3)].
    
    For any $\mathcal{E}\in\mathcal{C}^{uni}(X_K)$, consider the Jordan H\"older filtration of $\mathcal{E}$ in $\mathcal{C}^{uni}(X_K)$:
    $$0\hookrightarrow \mathcal{E}_1\hookrightarrow \cdots\hookrightarrow\mathcal{E}_n=\mathcal{E},$$
    where $\mathcal{E}_{i+1}/\mathcal{E}_i\cong\mathcal{O}_{X_K}$ for any $i$. Then we have the following exact sequence
    $$0\rightarrow \mathcal{O}_{X}\otimes_k K\rightarrow \mathcal{E}_2\rightarrow \mathcal{O}_{X}\otimes_k K\rightarrow 0.$$
    According to Lemma~\ref{subquotient}, $\mathcal{E}_2$ is a subquotient of $E'_2\otimes_k K$ for some $E'_2\in \mathcal{C}^{uni}(X)$. By [Proposition~\ref{fieldgeneral}, (3)], $\mathcal{E}_2\hookrightarrow E_2\otimes_k K$ in $\mathcal{C}^{uni}(X_K)$ for some $E_2\in\mathcal{C}^{uni}(X)$. Consider the following exact sequence
    $$0\rightarrow \mathcal{E}_{2}\rightarrow \mathcal{E}_3\rightarrow \mathcal{O}_{X_K}\rightarrow 0,$$
    we have the following pushout diagram
    \[
        \begin{tikzcd}
            0\arrow[r]&\mathcal{E}_{2}\arrow[r]\arrow[d,hook]&\mathcal{E}_3\arrow[r]\arrow[d,hook]&\mathcal{O}_{X_K}\arrow[r]\arrow[d,"\cong"]&0\\
            0\arrow[r]&E_2\otimes_k K\arrow[r]&\mathcal{E}'\arrow[ul, phantom, "\lrcorner", very near start]\arrow[r]&\mathcal{O}_X\otimes_k K\arrow[r]&0
        \end{tikzcd}
    \]
    According to [Proposition~\ref{fieldgeneral}, (3)] and Lemma~\ref{subquotient}, $\mathcal{E}_3$ is a subobject of $E_3\otimes_k K$ for some $E_3\in \mathcal{C}^{uni}(X)$. Repeating this method, we have $\mathcal{E}$ is a subobject of $E\otimes_k K$ for some $E\in\mathcal{C}^{uni}(X)$. Hence the natural homomorphism $\pi^{uni}(X_K,x_K)\rightarrow \pi^{uni}(X,x)_K$ is an isomorphism according to Theorem~\ref{fieldgeneralcriteria}.
\end{proof}

\begin{Proposition}\label{galoisinvariance}
    Let $k$ be a field, $K/k$ a finite Galois extension, $X$ a geometrically reduced connected scheme proper over $k$, $*\in\{S,N,Loc,\acute{e}t\}$. Then for any $\mathcal E\in\mathcal C^*(X_K)$ and any automorphism $\tilde g:X_K\rightarrow X_K$ induced by $g\in\Gal(K/k)$, we have $\tilde g^*\mathcal E\in\mathcal C^*(X_K)$.
\end{Proposition}

\begin{proof}
    For $*=S$: It follows immediately by Lemma~\ref{pullbacknef}.
    
    For $*=N$: If $\mathcal{E}$ is a finite bundle on $X_K$, then there exist distinct $f(t),h(t)\in\mathbb{N}[t]$ such that $f(\mathcal{E})\cong h(\mathcal{E})$. Then
    $$f(\tilde{g}^*\mathcal{E})\cong \tilde{g}^*f(\mathcal{E})\cong \tilde{g}^*h(\mathcal{E})\cong h(\tilde{g}^*\mathcal{E}),$$
    i.e. $\tilde{g}^*\mathcal{E}$ is a finite bundle on $X_K$. For any $\mathcal{E}\in\mathcal{C}^N(X_K)$, there exists a finite bundle $\mathcal{F}$ on $X_K$ such that $\mathcal{E}$ is a subquotient of $\mathcal{F}$ in $\mathcal{C}^{NF}(X_K)$. By Lemma~\ref{pullbacknef}, $\tilde{g}^*\mathcal{E}$ is a subquotient of $\tilde{g}^*\mathcal{F}$ in $\mathcal{C}^{NF}(X_K)$, so $\tilde{g}^*\mathcal{E}\in\mathcal{C}^{N}(X_K)$.

    For $*=Loc$: Consider the following commutative diagram
    \[
        \begin{tikzcd}
            X_K\arrow[r,"F_{X_K}"]\arrow[d,"\tilde{g}"]& X_K\arrow[d,"\tilde{g}"]\\
            X_K\arrow[r,"F_{X_K}"]&X_K
        \end{tikzcd}
    \]
    Then since there exists positive integer $n$ such that $F_{X_K}^{n*}\mathcal{E}\cong \mathcal{O}_{X_K}^{\oplus r}$, we have 
    $$F_{X_K}^{n*}\tilde{g}^*\mathcal{E}\cong \tilde{g}^*F_{X_K}^{n*}\mathcal{E}\cong \tilde{g}^*\mathcal{O}_{X_K}^{\oplus r}\cong \mathcal{O}_{X_K}^{\oplus r},$$
    i.e. $\tilde{g}^*\mathcal{E}\in\mathcal{C}^{Loc}(X_K)$.

% Target: galoisinvariance, *=etale. Source: separableet (1).

    For $*=\acute{e}t$: There exists finite \'etale covering $\pi_K:P\rightarrow X_K$ such that $\pi_K^*\mathcal{E}\cong \mathcal{O}_{P}^{\oplus n}$. Consider the following Cartesian diagram
    \[
        \begin{tikzcd}
            P\times_{X_K} X_K\arrow[r,"\pi"]\arrow[d,"\bar{g}"]&X_K\arrow[d,"\tilde{g}"]\\
            P\arrow[r,"\pi_K"]&X_K
        \end{tikzcd}
    \]
    Then $\pi:P\times_{X_K} X_K\rightarrow X_K$ is a finite \'etale covering and $\pi^* \tilde{g}^*\mathcal{E}\cong \bar{g}^*\pi_K^*\mathcal{E}\cong \mathcal{O}_{P\times_{X_K} X_K}^{\oplus n}$, hence $\tilde{g}^*\mathcal{E}\in\mathcal{C}^{\acute{e}t}(X_K)$.

    For $*\in \{N,Loc,\acute{e}t\}$. Take $\mathcal E\in \overline{\mathcal{C}^{*}(X_K)}$, and choose finite filtration with subquotients $\operatorname{gr}_i(\mathcal E)$ in $\mathcal C^*(X_K)$. Then $\tilde{g}^*\mathcal E$ admits a filtration with subquotients $\tilde{g}^*\operatorname{gr}_i(\mathcal E)\in\mathcal{C}^*(X_K)$. Hence $\tilde{g}^*\mathcal E\in\overline{\mathcal{C}^{*}(X_K)}$.
\end{proof}

\begin{Theorem}\label{algebraicbasechange}
    Let $k$ be a field, $K/k$ an algebraic extension,
    $X$ a geometrically reduced connected scheme proper over $k$,
    $x_K\in X_K(K)$ lying over $x\in X(k)$. Then 
    \begin{enumerate}
        \item Let $*\in\{S,N,\acute{e}t,EN,E\acute{e}t\}$. The natural homomorphism $\pi^*(X_K,x_K)\rightarrow\pi^*(X,x)_K$ is an isomorphism.
        \item Let $K/k$ be separable and $*\in\{Loc,ELoc\}$. The natural homomorphism $\pi^*(X_K,x_K)\rightarrow\pi^*(X,x)_K$ is an isomorphism.
    \end{enumerate}
\end{Theorem}

\begin{proof}
    (1)  If $K/k$ is finite separable, then for $*\in \{S,N,Loc,\acute{e}t\}$, there is a finite Galois extension $M/K/k$ such that both $M/K$ and $M/k$ are Galois. Then by Proposition~\ref{generalvbproperty}, Proposition~\ref{galoisinvariance}, Theorem~\ref{fieldfinite} and Proposition~\ref{fieldfiniteequiv}, we have
    \[
        \begin{tikzcd}
            \pi^*(X_M,x_M)\arrow[r,"\cong"]\arrow[rr,bend left=15,"\cong"]& \pi^*(X_K,x_K)_M\arrow[r]&\pi^*(X,x)_M
        \end{tikzcd}
    \]
    It follows that $\pi^*(X_K,x_K)_M\xrightarrow{\cong}({\pi^*(X,x)_K})_M$. Then by Lemma~\ref{isofield}, the natural homomorphism $\pi^*(X_K,x_K)\rightarrow\pi^*(X,x)_K$ is an isomorphism.
    
    If $K/k$ is purely inseparable of degree $d^n>0$, where $d=\characteristic(k)$, then for $*\in \{S,N\}$, since $\mathcal{C}^*(X_K)$ is saturated, the natural homomorphism $\pi^*(X_K,x_K)\rightarrow\pi^*(X,x)_K$ is an isomorphism by Proposition~\ref{fieldfiniteequiv} and Proposition~\ref{generalvbproperty}. For any $\mathcal{E}\in\mathcal C^{\acute{e}t}(X_K)$, there exists a finite \'etale covering $\pi_K:P\rightarrow X_K$ such that $\pi_K^*\mathcal E\cong \mathcal{O}_P^{\oplus r}$. By {\href{https://stacks.math.columbia.edu/tag/0BQN}{Tag 0BQN}}, there exists a finite \'etale covering $\pi:Q\rightarrow X$ such that the following diagram is cartesian
    \[
        \begin{tikzcd}
            P\arrow[r,"\pi_K"]\arrow[d,"p'"]& X_K\arrow[d,"p"]\\
            Q\arrow[r,"\pi"]&X
        \end{tikzcd}
    \]
    Then by cohomology and flat base change, we have 
    $$\pi^*p_*\mathcal E\cong p'_*\pi_K^*\mathcal E\cong p'_*\mathcal{O}_P^{\oplus r}\cong \mathcal O_Q^{\oplus{r[K:k]}}.$$
    Hence $p_*\mathcal E\in\mathcal C^{\acute{e}t}(X)$ and the natural homomorphism $\pi^{\acute{e}t}(X_K,x_K)\rightarrow \pi^{\acute{e}t}(X,x)_K$ is an isomorphism.

    If $K/k$ is finite, we have a tower of extension $K/k_{sep}/k$ where $k_{sep}$ is the separable closure of $k$ in $K$, then $K/k_{sep}$ is purely inseparable and $k_{sep}/k$ is separable. Then for $*\in\{S,N,\acute{e}t\}$, the natural homomorphism $\pi^*(X_K,x_K)\rightarrow \pi^*(X,x)_K$ is an isomorphism. 
    
    If $K/k$ is infinite, then for $*\in\{S,N,\acute{e}t\}$ and for any $\mathcal{E}\in \mathcal{C}^{*}(X_K)$, since $\mathcal{E}$ is a vector bundle, there exists finite extension $L/k$ contained in $K$ and $E\in\Vect(X_L)$ such that $\mathcal{E}\cong E\otimes_L K$. By Proposition~\ref{generalvbproperty}, we have $E\in\mathcal{C}^{*}(X_L)$. Since $L/k$ is finite, we have $\pi^*(X_L,x_L)\cong \pi^*(X,x)_L$, then
    $$E\in\mathcal{C}^{*}(X_L)=\Rep_L^f(\pi^*(X,x)_L).$$
    By Theorem~\ref{fieldgeneralcriteria}, there exists $F\in\mathcal{C}^{*}(X)$ such that $E\hookrightarrow F\otimes_k L$ and thus 
    $$\mathcal{E}\cong E\otimes_L K\hookrightarrow (F\otimes_k L)\otimes_L K\cong F\otimes_k K.$$
    Hence by Theorem~\ref{fieldgeneralcriteria}, the natural homomorphism $\pi^*(X_K,x_K)\rightarrow \pi^*(X,x)_K$ is an isomorphism. 
    
    The preceding argument proves the assertion for $*\in\{S,N,\acute{e}t\}$ and every algebraic extension $K/k$. Applying Proposition~\ref{saturationisomorphism}, we obtain the assertion for $*\in\{EN,E\acute{e}t\}$ as well.
    
    (2) If $K/k$ is infinite, then for any $\mathcal{E}\in \mathcal{C}^{Loc}(X_K)$, since $\mathcal{E}$ is a vector bundle, there exists finite separable extension $L/k$ contained in $K$ and $E\in\Vect(X_L)$ such that $\mathcal{E}\cong E\otimes_L K$. Then the natural homomorphism $\pi^{Loc}(X_K,x_K)\rightarrow\pi^{Loc}(X,x)_K$ is an isomorphism follows by the discussion in (1). The natural homomorphism $\pi^{ELoc}(X_K,x_K)\rightarrow\pi^{ELoc}(X,x)_K$ is an isomorphism by Proposition~\ref{saturationisomorphism}.
\end{proof}

\begin{Example}\label{puretranscendentalcounterexample}
    Let $k=\overline{\mathbb F}_5$, $K=k(t)$, where $t$ is transcendental over $k$, $C=\{y^2z=x^3\}\subseteq\mathbb P_k^2$, and $o=[0:1:0]\in C(k)$. For any $*\in\{S,N\}$, the natural homomorphism $\pi^*(C_K,o_K)\rightarrow\pi^*(C,o)_K$ is not an isomorphism.
\end{Example}

\begin{proof}
    The curve $C$ has a unique singular point $[0:0:1]$. For any field extension $F/k$ and any $a\in F$, let $P_a=[a:1:a^3]\in C_F(F)$. By \cite[Chapter~II, Example~6.11.4]{Har77}, we have an isomorphism $(F,+)\xrightarrow{\cong}\operatorname{Pic}^0(C_F), a\mapsto [\mathcal O_{C_F}([P_a]-[o_F])]$ compatible with field extensions. Consider the line bundle $\mathcal L=\mathcal O_{C_K}([P_t]-[o_K])$. Since $\characteristic k=5$, then we have  $\mathcal L^{\otimes5}\cong\mathcal O_{C_K}$. Hence $\mathcal L$ is finite, and 
    \[
        \mathcal L\in\mathcal C^N(C_K)\subseteq\mathcal C^{NF}(C_K).
    \]
    Suppose that $\mathcal L\cong L\otimes_k K$ for some line bundle $L$ on $C$, then $\deg L=0$. It follows that $L$ corresponds to some $a\in k$, and $L\otimes_k K$ corresponds to the same element $a\in K$. Since the class of $\mathcal L$ corresponds to $t$, this would imply $a=t$, contrary to $t\notin k$.

    Fix $*\in\{S,N\}$. Since $\mathcal L$ has rank one, it is a simple object of $\mathcal C^*(C_K)$. If the natural homomorphism $\pi^*(C_K,o_K)\to\pi^*(C,o)_K$ is an isomorphism, Proposition~\ref{fieldalgclosed}(2) would imply $\mathcal L\cong L\otimes_kK$ for some line bundle $L\in\mathcal C^*(C)$, contrary to the preceding paragraph.
\end{proof}

\begin{Example}\label{ex:bc-local-counterexample}
Let $k=\mathbb F_2(t)$, $K=k(s)$, $s^2=t$, $C$ be the smooth projective plane cubic over $k$ with affine equation $y^2+xy=x^3+t$, $o$ its point at infinity, $p:C_K\to C$ the projection and $F_{C_K}:C_K\rightarrow C_K$ the absolute Frobenius morphism. Then the natural homomorphisms $\pi^{Loc}(C_K,o_K) \rightarrow\pi^{Loc}(C,o)_K$ and $\pi^{ELoc}(C_K,o_K) \rightarrow\pi^{ELoc}(C,o)_K$ are not isomorphisms.
\end{Example}

\begin{proof}
The discriminant of $y^2+xy=x^3+t$ is $t\ne0$, so that $(C,o)$ is an elliptic curve. Let $P=(0,s)\in C_K(K)$, $\mathcal{L}=\mathcal O_{C_K}([P]-[o_K])$. The inverse of a point $(x,y)$ on $C$ is $(x,y+x)$. Hence $P=-P$ and $P\ne o_K$. It follows that $\mathcal{L}^{\otimes2}\cong\mathcal O_{C_K}$, $F_{C_K}^*\mathcal{L}\cong\mathcal O_{C_K}$. Thus $\mathcal L\in\mathcal C^{Loc}(C_K)$. By Theorem~\ref{algebraicbasechange}, $p_*\mathcal{L}\in \mathcal C^{NF}(C)$ and is of rank $2$. Lemma~\ref{inseparable-filtration} gives an exact sequence
\begin{equation}\label{eq:bc-local-self-extension}
  0\rightarrow \mathcal{L}\rightarrow p_*\mathcal{L}\otimes_k K
  \rightarrow \mathcal{L}\rightarrow0.
\end{equation}
Suppose that $M$ is a nonzero proper subobject of $p_*\mathcal{L}$ in $\mathcal C^{NF}(C)$. Then $M$ has rank $1$, and by uniqueness of factors of Jordan H\"older filtration of $p_*\mathcal{L}\otimes_k K$ we have $M\otimes_k K\cong \mathcal{L}$. The Abel map $C\to\operatorname{Pic}^0_{C/k}$, defined by $Q\mapsto\mathcal O_C(Q-o)$, is an isomorphism over $k$. If $\mathcal{L}\cong M\otimes_k K$ for a line bundle $M\in\mathcal C^{NF}(C)$, then the class of $M$ gives a $k$-point of $\operatorname{Pic}^0_{C/k}$ whose base change corresponds to $P$. This would imply $P\in C(k)$, contrary to $s\notin k$. Hence no such line bundle $M$ exists. Thus $p_*\mathcal{L}$ is irreducible in $\mathcal C^{NF}(C)$.

Schur's lemma implies that $\operatorname{End}_{\mathcal O_C}(p_*\mathcal{L})$ is a division algebra over $k$ and we have an injective $k$-algebra homomorphism $K\hookrightarrow \operatorname{End}_{\mathcal O_C}(p_*\mathcal{L})$. The fibre functor at
$o$ makes $p_*\mathcal{L}|_o$ a faithful $\operatorname{End}_{\mathcal O_C}(p_*\mathcal{L})$-module of dimension $2$ over $k$.
For any nonzero $v\in p_*\mathcal{L}|_o$, the map $\operatorname{End}_{\mathcal O_C}(p_*\mathcal{L})\to p_*\mathcal{L}|_o$, $a\mapsto av$, is injective, since every nonzero element of $\operatorname{End}_{\mathcal O_C}(p_*\mathcal{L})$ is invertible. Consequently,
\[
  2=[K:k]\leq\dim_k\operatorname{End}_{\mathcal O_C}(p_*\mathcal{L})\leq\dim_k(p_*\mathcal{L}|_o)=2,
\]
and hence $\operatorname{End}_{\mathcal O_C}(p_*\mathcal{L})=K$. By cohomology and flat base change, we have
\begin{align*}
  \operatorname{End}_{\mathcal O_{C_K}}(p_*\mathcal{L}\otimes_k K)
    &\cong \operatorname{End}_{\mathcal O_C}(p_*\mathcal{L})\otimes_k K\\
    &\cong K\otimes_k K
     \cong K[T-s]/\left((T-s)^2\right).
\end{align*}
This algebra has dimension $2$ over $K$. If \eqref{eq:bc-local-self-extension} were split, then $\operatorname{End}_{\mathcal O_{C_K}}(p_*\mathcal{L}\otimes_k K)\cong \operatorname{Mat}_2(K)$ is of dimension $4$. Therefore \eqref{eq:bc-local-self-extension} is non-split.

Consider the affine open cover $U=C_K\setminus\{o_K\}$, $W=C_K\setminus\{P\}$. The rational function $r=\frac{y}{x}$ is regular on $U\cap W$ and has a simple pole at $o_K$. Its \v{C}ech class $[r]\in H^1(C_K,\mathcal O_{C_K})$ is nonzero: if $r=f+g$ with $f\in\mathcal O(U)$ and $g\in\mathcal O(W)$, then $f$ would have a simple pole at $o_K$ and no other pole, contrary to $h^0(C_K,\mathcal O_{C_K}(o_K))=1$. Dividing $y^2+xy=x^3+t$ by $x^2$, we obtain $r^2=r+x+\frac{t}{x^2}$. Here $x$ is regular on $U$, and $t/x^2$ is regular on $W$, since $x$ has its only zero at $P$ and a pole of order $2$ at $o_K$. Thus $[r^2]=[r]$. Since $H^1(C_K,\mathcal O_{C_K})$ is one-dimensional, absolute Frobenius acts by $a[r]\mapsto a^2[r], a\in K$, which is injective. Tensoring \eqref{eq:bc-local-self-extension} with $\mathcal{L}^{-1}$, we obtain a non-split extension
\[
  0\rightarrow\mathcal O_{C_K}\rightarrow
  (p_*\mathcal{L}\otimes_k K)\otimes \mathcal{L}^{-1}
  \rightarrow\mathcal O_{C_K}\rightarrow0
\]
with nonzero class $a[r]\in H^1(C_K,\mathcal O_{C_K})$ for some $a\in K$. For every $n\geq1$, pulling back this extension by $F_{C_K}^n$ gives the nonzero extension class $a^{2^n}[r]$. Since $F_{C_K}^{n*}\mathcal{L}$ is trivial, $F_{C_K}^{n*}(p_*\mathcal{L}\otimes_k K)$ is a non-split extension of $\mathcal O_{C_K}$ by itself. Such an extension cannot have trivial middle term, hence $p_*\mathcal{L}\otimes_k K\notin\mathcal C^{\mathrm{Loc}}(C_K)$ and $p_*\mathcal{L}\notin\mathcal C^{\mathrm{Loc}}(C)$. Then it follows by Theorem~\ref{fieldfinite}.

Now suppose $p_*\mathcal{L}\in\overline{\mathcal C^{Loc}(C)}$, then there exists a filtration of $p_*\mathcal{L}$ with Frobenius trivial successive quotients. Since $p_*\mathcal{L}$ is irreducible in $\mathcal C^{NF}(C)$, we obtain $p_*\mathcal{L}\in\mathcal C^{Loc}(C)$, a contradiction. Hence $p_*\mathcal{L}\notin\overline{\mathcal C^{Loc}(C)}$. Then it follows by Theorem~\ref{fieldfinite}.
\end{proof}

\begin{Lemma}\label{NFirreduciblebasechange}
    Let $k$ be an algebraically closed field, $K/k$ an extension of fields,
    $X$ a reduced connected scheme proper over $k$, and
    $E$ a simple object of $\mathcal C^{NF}(X)$.
    Then $E\otimes_k K$ is a simple object of
    $\mathcal C^{NF}(X_K)$.
\end{Lemma}

\begin{proof}
    Suppose that there exists an exact sequence
    \[
        0\rightarrow\mathcal F\rightarrow E\otimes_k K
        \rightarrow\mathcal Q\rightarrow0
    \]
    in $\mathcal C^{NF}(X_K)$, where
    $0<\rk\mathcal F<\rk E$. By spreading out, there exists a finitely generated
    $k$-subalgebra $R\subseteq K$ such that, after replacing
    $R$ by a localization, the above sequence descends to
    an exact sequence of vector bundles on $X_R$:
    \[
        0\rightarrow\mathcal F_R\rightarrow E\otimes_k R
        \rightarrow\mathcal Q_R\rightarrow0.
    \]
    Since $k$ is algebraically closed, choose $s\in\Spec R(k)$. Restricting to the fibre over $s$, we obtain
    \[
        0\rightarrow\mathcal F_s\rightarrow E
        \rightarrow\mathcal Q_s\rightarrow0.
    \]
    Let $f:C\rightarrow X$ be any morphism from a smooth projective curve over $k$. Since $R$ is integral, the degree of the restriction of $f_R^*\mathcal F_R$ to the fibres of $C_R\rightarrow \Spec R$ is constant. Since $\mathcal F$
    is numerically flat, we have $\deg(f^*\mathcal F_s)=\deg(f_K^*\mathcal F)=0$. Moreover, $f^*E$ is semistable of degree zero.
    Therefore $f^*\mathcal F_s$ and $f^*\mathcal Q_s$
    are semistable of degree zero. By the curve criterion
    for numerical flatness, we obtain
    $\mathcal F_s,\mathcal Q_s\in\mathcal C^{NF}(X)$. Thus $\mathcal F_s$ is a nonzero proper subobject of
    $E$ in $\mathcal C^{NF}(X)$, a contradiction.
\end{proof}

\begin{Proposition}\label{algfaith}
    Let $k$ be an algebraically closed field, $K/k$ an extension of fields, $X$ a reduced connected scheme proper over $k$, $*\in\{S,N,Loc,\acute{e}t,uni, EN,ELoc,E\acute{e}t\}$, $x_K\in X_K(K)$ lying over $x\in X(k)$. Then the natural homomorphism $\pi^{*}(X_K,x_K)\rightarrow \pi^{*}(X,x)_K$ is faithfully flat.
\end{Proposition}

\begin{proof}
    By Lemma~\ref{NFirreduciblebasechange} and [Proposition~\ref{fieldalgclosed}, (1)], the base change homomorphism for the S-fundamental group scheme is faithfully flat. The other categories in the statement are Tannakian subcategories of the category of numerically flat bundles, stable under subquotients. For their saturation, this follows by using the induced filtrations. The assertion now follows from Propositions~\ref{transurjective} and~\ref{generalvbproperty}.
\end{proof}

In characteristic zero, base change for the \'etale fundamental gerbe under arbitrary field extension was proved by Bresciani \cite[Proposition~A.23]{Bre21}. For a pointed scheme, this yields the corresponding statement for the \'etale fundamental group scheme. The proof below treats proper schemes in arbitrary characteristic using the Tannakian criterion and the invariance of the geometric \'etale fundamental group.

\begin{Lemma}[{\cite[Lemma~5.6.7]{Sza09}}]\label{etalegroupalgclosed}
    Let $K/k$ be an extension of algebraically closed fields, $X$ a reduced connected scheme proper over $k$, $x_K\in X_K(K)$ lying over $x\in X(k)$. Then the natural map $\pi^{\acute{e}t}_1(X_K, x_K)\rightarrow \pi^{\acute{e}t}_1(X,x)$ is an isomorphism.
\end{Lemma}

\begin{Proposition}\label{etalealgclosed}
    Let $K/k$ be an extension of fields, $X$ a geometrically reduced connected scheme proper over $k$, $x_K\in X_K(K)$ lying over $x\in X(k)$. Then
    \begin{enumerate}
        \item For $*\in\{\acute{e}t,E\acute{e}t\}$, the natural homomorphism $\pi^{*}(X_K, x_K)\rightarrow \pi^{*}(X,x)_K$ is an isomorphism.
        \item Let $k$ be an algebraically closed field. For any simple object $\mathcal{E}\in\mathcal{C}^{\acute{e}t}(X_K)$, there exists $E\in\mathcal{C}^{\acute{e}t}(X)$ such that $\mathcal{E}\cong E\otimes_k K$.
    \end{enumerate}
\end{Proposition}

\begin{proof}
    (1) Let ${\overline{K}}$ be an algebraic closure of $K$, and $\bar k$ the algebraic closure of $k$ in ${\overline{K}}$. Then $\bar k/k$ is algebraic and ${\overline{K}}/\bar k$ is an extension of algebraically closed fields. Since $\bar{k}$ is algebraically closed, every finite \'etale $\bar{k}$-group scheme is constant; see \cite[Chapter II, \S~5, 1.7]{DeGa70}. Therefore, for every finite \'etale $\bar{k}$-group scheme $G$, the natural map $G(\bar{k})\rightarrow G_{\overline{K}}({\overline{K}})$ is an isomorphism. Since $\pi^{\acute{e}t}(X_{\bar{k}},x_{\bar{k}})$ is an inverse limit of finite \'etale $\bar{k}$-group schemes, passing to inverse limits gives a natural isomorphism of profinite groups
    \[
        \pi^{\acute{e}t}(X_{\bar{k}},x_{\bar{k}})({\bar{k}})\cong \pi^{\acute{e}t}(X_{\bar{k}},x_{\bar{k}})({\overline{K}}) \cong \pi^{\acute{e}t}(X_{\bar{k}},x_{\bar{k}})_{\overline{K}}({\overline{K}}).
    \]
    Note that $\pi^{\acute{e}t}(X_{\overline{K}},x_{\overline{K}})({\overline{K}}) \cong \pi_1^{\acute{e}t}(X_{\overline{K}},x_{\overline{K}})$, $\pi^{\acute{e}t}(X_{\bar{k}},x_{\bar{k}})(\bar{k}) \cong \pi_1^{\acute{e}t}(X_{\bar{k}},x_{\bar{k}})$, so we have the following commutative diagram of profinite groups:
    \[
        \begin{tikzcd}
            \pi^{\acute{e}t}(X_{\overline{K}},x_{\overline{K}})({\overline{K}})
            \arrow[r]
            \arrow[d,"\cong"]
            &
            \pi^{\acute{e}t}(X_{\bar{k}},x_{\bar{k}})_{\overline{K}}({\overline{K}})
            \arrow[d,"\cong"]
            \\
            \pi_1^{\acute{e}t}(X_{\overline{K}},x_{\overline{K}})
            \arrow[r,"\cong"]
            &
            \pi_1^{\acute{e}t}(X_{\bar{k}},x_{\bar{k}}).
        \end{tikzcd}
    \]
    Hence the top arrow is an isomorphism of profinite groups. Both $\pi^{\acute{e}t}(X_{\overline{K}},x_{\overline{K}})$ and $\pi^{\acute{e}t}(X_{\bar{k}},x_{\bar{k}})_{\overline{K}}$ are profinite \'etale ${\overline{K}}$-group schemes. It follows that the natural homomorphism $\pi^{\acute{e}t}(X_{\overline{K}},x_{\overline{K}}) \rightarrow \pi^{\acute{e}t}(X_{\bar{k}},x_{\bar{k}})_{\overline{K}}$ is an isomorphism. For any $\mathcal E\in\mathcal C^{\acute{e}t}(X_K)$, we have $\mathcal E\otimes_K{\overline{K}}\in\mathcal C^{\acute{e}t}(X_{\overline{K}})$. By Theorem~\ref{fieldgeneralcriteria}, there exists $\mathcal F\in\mathcal C^{\acute{e}t}(X_{\bar k})$ such that $\mathcal E\otimes_K{\overline{K}} \hookrightarrow \mathcal F\otimes_{\bar k}{\overline{K}}$. Since $\bar k/k$ is algebraic, by Theorem~\ref{fieldgeneralcriteria} and Theorem~\ref{algebraicbasechange}, there exists $E\in\mathcal C^{\acute{e}t}(X)$ such that $\mathcal F\hookrightarrow E\otimes_k\bar k$. Hence we obtain an injection $f:\mathcal E\otimes_K{\overline{K}} \hookrightarrow E\otimes_k{\overline{K}}$. Note that
    $$
        \Hom_{\mathcal O_{X_{\overline{K}}}}
        (\mathcal E\otimes_K{\overline{K}},E\otimes_k{\overline{K}})
        \cong
        \Hom_{\mathcal O_{X_K}}
        (\mathcal E,E\otimes_kK)\otimes_K{\overline{K}}.
    $$
    Then we have $f=\sum_{i=1}^{m}f_i\otimes a_i$, where $f_i\in\Hom_{\mathcal O_{X_K}}(\mathcal E,E\otimes_kK)$ and $a_i\in{\overline{K}}$. Consider the morphism
    $$
        g:=(f_1,\ldots,f_m):
        \mathcal E\rightarrow(E\otimes_kK)^{\oplus m}.
    $$
    After tensoring with ${\overline{K}}$, the morphism $f$ factors through
    $$
        \mathcal E\otimes_K{\overline{K}}
        \xrightarrow{g\otimes_K{\overline{K}}}
        (E\otimes_k{\overline{K}})^{\oplus m}
        \xrightarrow{(a_1,\ldots,a_m)}
        E\otimes_k{\overline{K}}.
    $$
    Since $f$ is injective, $g\otimes_K{\overline{K}}$ is injective. Since ${\overline{K}}/K$ is faithfully flat, $g$ is also injective. Therefore, we have $\mathcal E\hookrightarrow(E\otimes_kK)^{\oplus m}\cong E^{\oplus m}\otimes_k K$. Since $E^{\oplus m}\in\mathcal C^{\acute{e}t}(X)$, the natural homomorphism $\pi^{\acute{e}t}(X_K, x_K)\rightarrow \pi^{\acute{e}t}(X,x)_K$ is an isomorphism by Theorem~\ref{fieldgeneralcriteria}. By Proposition~\ref{saturationisomorphism}, the natural homomorphism $\pi^{E\acute{e}t}(X_K, x_K)\rightarrow \pi^{E\acute{e}t}(X,x)_K$ is an isomorphism.
    
    (2) It follows by (1) and [Proposition~\ref{fieldalgclosed}, (2)].
\end{proof}

\begin{Lemma}[{\cite[Theorem~1.1]{BiDu07}}]\label{etaletrivialequiv}
    Let $k$ be an algebraically closed field of characteristic $p>0$, $X$ a smooth projective variety over $k$, $F_X:X\rightarrow X$ the absolute Frobenius morphism, $E\in\Vect(X)$. Then 
    \begin{enumerate}
        \item If there exists integer $n>0$ such that $F_X^{n*}E\cong E$, then $E\in\mathcal{C}^{\acute{e}t}(X)$.
        \item If $E$ is stable, then there exists integer $n>0$ such that $F_X^{n*}E\cong E$ iff $E\in\mathcal{C}^{\acute{e}t}(X)$.
    \end{enumerate}
\end{Lemma}

\begin{Corollary}
    Let $K/k$ be an extension of algebraically closed fields of characteristic $p>0$, $X$ a smooth projective variety over $k$, $F_X:X\rightarrow X$ the absolute Frobenius morphism, $\mathcal{E}\in\Vect(X_K)$ a stable vector bundle. If there exists integer $n>0$ such that $F^{n*}_{X_K}\mathcal{E}\cong \mathcal{E}$, then there exists $E\in\mathcal{C}^{\acute{e}t}(X)$ such that $\mathcal{E}\cong E\otimes_k K$ and $F_X^{n*}E\cong E$.
\end{Corollary}

\begin{proof}
    According to Lemma~\ref{etaletrivialequiv}, $F^{n*}_{X_K}\mathcal{E}\cong \mathcal{E}$ implies that $\mathcal{E}\in\mathcal{C}^{\acute{e}t}(X_K)$. Since $\mathcal{E}$ is stable, we have $\mathcal{E}$ is simple in $\mathcal{C}^{\acute{e}t}(X_K)$. By [Proposition~\ref{fieldalgclosed}, (2)] and Proposition~\ref{etalealgclosed}, there exists $E\in\mathcal{C}^{\acute{e}t}(X)$ such that $\mathcal{E}\cong E\otimes_k K$. Note that 
    $$\begin{aligned}
        (F_{X}^{n*}E)\otimes_k K\cong F_{X_K}^{n*}(E\otimes_k K)\cong F_{X_K}^{n*}\mathcal{E}\cong \mathcal{E}\cong E\otimes_k K,
    \end{aligned}$$
    i.e. $(F_{X}^{n*}E)\otimes_k K\cong E\otimes_k K$. Note that we have $F_{X}^{n*}E$ and $E\in\mathcal{C}^{NF}(X)$. Then we have $F_{X}^{n*}E\cong E$ according to Lemma~\ref{NoetherDeuring}.
\end{proof}

In general, the base change of fundamental group schemes may fail under algebraically closed field extension. A counterexample for the base change of Local fundamental group scheme is stated by Mehta and Subramanian\cite{MeSu02}, saying that for an integral projective curve over an algebraically closed field $k$ with at least one cuspidal singularity, such base change could not hold since a Frobenius trivial line bundles on $X_K$ could not descend to a Frobenius trivial line bundle on $X$. Pauly\cite{Pau07} stated a smooth counterexample saying that there exists stable Frobenius trivial bundles of rank 2 on $X_K$ but not defined on $X$. St\"abler \cite{Sta14} constructed smooth projective counterexamples in every positive characteristic.

\begin{Corollary}\label{5.47}
    Let $k$ be an algebraically closed field of characteristic $p>0$, $K/k$ an extension of fields, $X$ a reduced connected scheme proper over $k$, $x_K\in X_K(K)$ lying over $x\in X(k)$. If there exists a simple object $\mathcal{E}_0\in\mathcal{C}^{Loc}({X_K})$ such that $\mathcal{E}_0\not\cong E\otimes_k K$ for any $E\in\mathcal{C}^{Loc}({X})$, then the natural homomorphism $\pi^{*}(X_K,x_K)\rightarrow \pi^{*}(X,x)_K$ is not an isomorphism, where $*\in \{S,N,ELoc\}$.
\end{Corollary}

\begin{proof}
    Note that for any $E\in\Vect(X)$, $E\in\mathcal{C}^{Loc}(X)$ iff $E\otimes_k K\in\mathcal{C}^{Loc}(X_K)$ by Proposition~\ref{generalvbproperty}. According to Proposition~\ref{algclosednotiso}, the natural homomorphism $\pi^{*}(X_K,x_K)\rightarrow \pi^{*}(X,x)_K$ is not an isomorphism, where $*\in \{S,N,ELoc\}$.
\end{proof}

We also have a counterexample for base change of S-fundamental group scheme in characteristic $0$.

\begin{Proposition}\label{containSnobasechange}
    Let $K/k$ be a proper extension of algebraically closed fields of characteristic $0$, $X$ a smooth projective curve over $k$ with genus $g\geq 2$, $x_K\in X_K(K)$ lying over $x\in X(k)$. Then the natural homomorphism $\pi^S(X_K,x_K)\rightarrow \pi^S(X,x)_K$ is faithfully flat, but not an isomorphism.
\end{Proposition}

\begin{proof}
    The natural homomorphism $\pi^S(X_K,x_K)\rightarrow \pi^S(X,x)_K$ is faithfully flat by Proposition~\ref{algfaith}. Since the characteristic of $k$ is 0, a vector bundle being strongly semistable is equivalent to being semistable. According to \cite[Theorem~2.13]{FuLa22}, the stable numerically flat bundles on $X_K$ are precisely the stable vector bundles of degree zero.
    Consider the coarse moduli space $\mathcal M_X^s(2,0)$ of stable vector bundles of rank $2$ and degree $0$ on $X$. It is a smooth quasi-projective variety of dimension $\dim\mathcal M_X^s(2,0)=4(g-1)+1=4g-3>0$ and we have $\mathcal M_{X_K}^s(2,0)\cong \mathcal M_X^s(2,0)\times_k K$. Choose an integral curve $C\subseteq\mathcal M_X^s(2,0)$. Since $k$ is algebraically closed and $K\neq k$, the field $K$ contains an element transcendental over $k$. Since $K$ is algebraically closed, there exists a $k$-embedding $k(C)\lhook\joinrel\rightarrow K$. This embedding induces a $K$-valued point $\Spec K\rightarrow \Spec k(C)\rightarrow C\rightarrow \mathcal M_X^s(2,0)$ and we denote this point by $z$. Its image in $\mathcal M_X^s(2,0)$ is the generic point of $C$, which is not closed. Since every $k$-valued point of $\mathcal M_X^s(2,0)$ has closed image, $z$ does not belong to the image of $\mathcal M_X^s(2,0)(k)$.

    As $K$ is algebraically closed, the coarse moduli property implies that $z$ is represented by a stable vector bundle $\mathcal E$ of rank $2$ and degree $0$ on $X_K$. If $\mathcal E\cong E\otimes_k K$ for some $E\in\Vect(X)$, then $E$ is stable of rank $2$ and degree $0$. Consequently, its moduli point $[E]\in\mathcal M_X^s(2,0)(k)$ maps to $z$, a contradiction. Thus $\mathcal E$ is a simple object of $\mathcal C^{NF}(X_K)$ which does not descend to an object of $\mathcal C^{NF}(X)$. By [Proposition~\ref{fieldalgclosed}, (2)], the natural homomorphism $\pi^S(X_K,x_K)\rightarrow \pi^S(X,x)_K$ cannot be an isomorphism.
\end{proof}
\noindent\textbf{AI disclosure.}
The authors used ChatGPT to assist in constructing the counterexamples in Remark~\ref{etnosaturated}, Example~\ref{puretranscendentalcounterexample} and Example~\ref{ex:bc-local-counterexample}. Apart from this author-directed assistance, all ideas and arguments are due to the authors. The authors independently verified all arguments and take full responsibility for the paper.

\section*{Acknowledgements}
We would like to thank Adrian Langer and Lei Zhang for their enlightening comments and helpful conversations. The authors are partially supported by Applied Basic Research Programs of Science and Technology Commission Foundation of Shanghai Municipality(22JC1402700), National Natural Science Foundation of China(Grant No. 12171352).

\end{document}